\documentclass{amsart}
\usepackage{amsmath,amsxtra,amssymb,latexsym, amscd,amsthm}
\usepackage{graphicx}

\usepackage{pgf,tikz}
\usetikzlibrary{arrows}
\usepackage[backref]{hyperref}

\usepackage{pgf,tikz}
\usetikzlibrary{arrows}
\usepackage[mathscr]{eucal}

\DeclareMathOperator{\Div}{div}

\DeclareMathOperator{\dist}{dist}
\DeclareMathOperator{\Gr}{Gr}
\DeclareMathOperator{\Ric}{Ric}
\DeclareMathOperator{\Area}{\mathcal{A}}
\DeclareMathOperator{\const}{const}

\begin{document}

\title[The Dirichlet problem for the minimal surface equation in ${\rm Sol_3}$]{The Dirichlet problem for the minimal\\surface equation in $\rm S\lowercase{ol}_3$,\\
with possible infinite boundary data}

\author{M\lowercase{inh} H\lowercase{oang} NGUYEN}


\subjclass[2000]{Primary: 53A10, 53C42; Secondary: 53A35, 53B25}

\keywords{Minimal surface, Killing graph, Dirichlet problem, Jenkins-Serrin type theorem}


\address{Laboratoire \'Emile Picard, UMR 5580, 
 Universit\'e Paul Sabatier, 
118 route de Narbonne, 31062 Toulouse cedex 04, France}

\address{Laboratoire d'Analyse et de Math\'ematiques Appliqu\'ees, UMR 8050,  
Universit\'e Paris-Est, Cit\'e Descartes
 5 bd Descartes, Champs-sur-Marne, 77454 Marne-la-Vallée cedex 2, France}

\email{minh-hoang.nguyen@math.univ-toulouse.fr}


\begin{abstract}
In this paper, we study the Dirichlet problem for the minimal surface equation in  ${\rm Sol_3}$ with possible infinite boundary data, where ${\rm Sol_3}$ is the non-abelian solvable $3$-dimensional Lie group equipped with its usual left-invariant metric that makes it into a model space for one of the eight Thurston geometries.
Our main result is a Jenkins-Serrin type theorem which establishes necessary and sufficient conditions for the existence and uniqueness of certain minimal Killing graphs with a non-unitary Killing vector field in ${\rm Sol_3}$.
\end{abstract}

\maketitle


\theoremstyle{plain}
\newtheorem{dl}{Theorem}[section]
\newtheorem{bd}[dl]{Lemma}
\newtheorem{md}[dl]{Proposition}
\newtheorem{hq}[dl]{Corollary}
\newtheorem{trh}{Case}[section]
\newtheorem{kd}{Assertion}[section]

\theoremstyle{definition}
\newtheorem{dn}[dl]{Definition}
\newtheorem{cy}[dl]{Remark}

\numberwithin{equation}{section}
\numberwithin{figure}{section}

\setcounter{tocdepth}{1}

\newcommand{\dsf}{\mathsf{d}}

\newcommand{\rr}{\mathbb R}
\newcommand{\hh}{\mathbb H}
\newcommand{\dhr}{\partial}
\newcommand{\bn}{\overline{\nabla}}
\newcommand{\sol}{{\rm Sol_3}}

\newcommand{\Pcal}{\mathcal{P}}
\newcommand{\Acal}{\mathcal{A}}

\newcommand{\Ucal}{\mathcal{U}}
\newcommand{\Vcal}{\mathcal{V}}

\newcommand{\Dcal}{\mathcal{D}}
\newcommand{\Dbb}{\mathbb{D}}

\newcommand{\euc}{{\rm euc}}
\newcommand{\hyp}{{\rm hyp}}

\newcommand{\vh}[1]{\left\langle #1\right\rangle}
\newcommand{\td}[1]{\left\| #1\right\|}

\newcommand{\bvc}{\partial_\infty'\hh^2}
\newcommand{\Mcal}{\mathfrak{M}}
\newcommand{\Hcal}{\mathcal{H}}
\newcommand{\Ngoac}[1]{\left(#1\right)}


\section{Introduction}\label{sec1}
In \cite{JS66}, Jenkins and Serrin considered bounded domains $\Omega\subset\rr^2$, with $\dhr\Omega$ composed of straight line segments and convex arcs. They found necessary and sufficient conditions on the lengths of the sides of inscribed polygons, which guarantee the existence of a minimal graph over $\Omega$, taking certain prescribed values (in $\rr\cup\{\pm\infty\}$) on the components of $\dhr\Omega$. 
Perhaps the simplest example is $\Omega$ with a geodesic triangle with boundary data zero on two sides and $+\infty$ on the third side. The conditions of Jenkins-Serrin reduce to the triangle inequality here and the solution exists. It was discovered by Scherk in 1835. This also works on a parallelogram with sides of equal length, with data $+\infty$ on opposite two sides and $-\infty$ on the other two sides. This solution was also found by Scherk.

In recent years there has been much activity on this Dirichlet problem in $M^2\times\rr$ where $M$ is a two dimensional Riemannian manifold (see \cite{CR10, NR02, Pin09, You10}). When $M$ is the hyperbolic plane $\hh^2$, there are  non-compact domains for which this problem has been solved, and interesting applications have been obtained (see \cite{CR10,GR10,MRR11}). 
In the previous cases, authors considered the Killing graphs where the Killing vector field is unitary.

The purpose of this paper is to consider the problem of type Jenkins-Serrin on bounded domains and some unbounded domains in $\sol$ which is a three-dimensional homogeneous Riemannian manifold can be viewed as $\rr^3$ endowed with the Riemannian metric
\begin{equation}
\dsf s^2=e^{2x_3}\dsf x_1^2+e^{-2x_3}\dsf x_2^2+\dsf x_3^2.
\end{equation}
The change of coordinates
\begin{equation*}
x:=x_2,\qquad y:= e^{x_3},\qquad t:=x_1,
\end{equation*}
turns this model into $\sol=\{(x,y,t)\in\rr^3:y\ge 0 \}$ with the Riemannian  metric
\begin{equation}
\dsf s^2=\dfrac{\dsf x^2+\dsf y^2}{y^2}+y^2\dsf t^2.
\end{equation}
\\
 By using the Poincar\'e half-plane model, $\sol$ has the form of a warped product  $\sol=\hh^2\times_y\rr$.

For every function $u$ of class $C^2$ defined on the domain $\Omega\subset\hh^2$, we denote by $\Gr(u)=\{(p,t)\in\sol:p\in\Omega,t=u(p)\}$  a surface in $\sol$ and is called $\dhr_t$-graph of $u$. $\Gr(u)$ is a minimal surface if and only if $u$ satisfies the equation
 (see Proposition \ref{equ-minimale})
\begin{equation}\label{intro1}
\Mcal u:=\Div\left(\frac{y^2\nabla u}{\sqrt{1+y^2\td{\nabla u}^2}}\right)=0.
\end{equation}
We will consider the case that the boundary $\dhr\Omega$ is composed of the families  of "convex" arcs $\{A_i\}$,  $\{B_j\}$ and  $\{C_k\}$. We give necessary and sufficient conditions on the geometry of the domain $\Omega$ which assure the existence of a minimal solution $u$ defined in $\Omega$ and $u$ assumes the value $+\infty$ on each $A_i$, $-\infty$ on each $B_j$ and prescribed
continuous data on each $C_k$.

We see that $\dhr_t$  is Killing and normal to the plane $\hh^2$. A special point of the problem is that the vector field $\dhr_t$ is not unitary. 
The important point to note here is that when $\gamma$ is a curve in $\hh^2$, if $\gamma$ is a geodesic of $\hh^2$, the surface $\gamma\times\rr$ is no longer minimal in this warped product Riemannian manifold $\sol$. 
Instead of this, $\gamma\times\rr$ is minimal in $\sol$ if and only if $\gamma$ is an Euclidean geodesic (see Corollary \ref{cor2.2}). Hence, these 
Euclidean geodesics will play an important role in our problem.
Moreover, because of the non-unitary field $\dhr_t$, we don't use the hyperbolic length to state our problem. 
In $M^2\times\rr$ the length of a compact curve $\gamma\subset M^2$ is just the area of $\gamma\times[0,1]$ in which we are interested. 
However, for a curve $\gamma\in\hh^2$, the area calculated in $\sol$ of $\gamma\times[0,1]$ is the Euclidean length of $\gamma$ (see Proposition \ref{pro2.3}).

The problem of type Jenkins-Serrin is also solved for some unbounded domains. The main idea in \cite{CR10} is to approximate an unbounded domain $\Omega$ by a sequence bounded domain $\Omega_n$ by cutting $\Omega$ with horocycles.

In our case, we use the Euclidean geodesics, Euclidean length instead of the geodesics and the hyperbolic length, so we can't use the horocycle of $\hh^2$ to consider the problem de type Jenkins-Serrin on an unbounded domain. However, we can generalize the previous result for some unbounded domains by defining the flux for the non-compact arcs instead of using the horocycles. Our main result (Theorem \ref{J-S type1}) may be stated as follows.

\bigskip

\noindent\textbf{Theorem.}\textit{\quad Let $\Omega$ be a Scherk domain in $\hh^2$ with the families of Euclidean geodesic arcs $\{A_i \},\{B_i \}$ and of Euclidean mean convex arcs $\{C_i \}$.
\begin{itemize}
\item[ (i)] If the family $\{C_i\}$ is non-empty, there exists a solution to the Dirichlet problem on $\Omega$ if and only if
\begin{equation*}
2a_\euc(\Pcal)<\ell_\euc(\Pcal),\qquad 2b_\euc(\Pcal)<\ell_\euc(\Pcal)
\end{equation*}
for every Euclidean polygonal domain inscribed in $\Omega$. Moreover, such a solution is unique if it exists.
\item[ (ii)] If the family $\{ C_i\}$ is empty, there exists a solution to the Dirichlet problem on $\Omega$ if and only if
\begin{equation*}
a_\euc(\Pcal)=b_\euc(\Pcal)
\end{equation*}
when $\Pcal=\Omega$ and the inequalities in (i) hold for all other Euclidean polygonal domains inscribed in $\Omega$. Such a solution is unique up to an additive constant, if it exists.\end{itemize}
}
\bigskip

We will have similar result for the Dirichlet problem for the minimal surface equation in $\sol$ with respect to $\dhr_x$-graph. In the case of $\dhr_y$-graph ($\dhr_y$ is not a Killing vector field),
Ana Menezes solved  on some "small" squares in the $(x,t)$-plane with data $+\infty$ on opposite two sides and $-\infty$ on the other two sides (see \cite[Theorem 2]{Men13}).

We have organized the contents as follows: 
In Section \ref{sec2}, we will review some of the standard facts on $\sol$ and establish minimal surface equations. 
Section \ref{sec3} will prove the maximum principle for the minimal surface equations, show the existence of solutions. 
A local Scherk surface in $\sol$ will be constructed in section \ref{sec4}. Sections \ref{sec5} will be devoted to proving the monotone convergence theorem and describing the divergence set. Our main results are stated and proved in Section \ref{sec6}.

\bigskip

\tableofcontents

\section{Preliminaries}\label{sec2}

\subsection{A model of $\sol$}

The three-dimensional homogeneous Riemannian manifold $\sol$ can be viewed as $\rr^3$ endowed with the Riemannian metric
\begin{equation}
\dsf s^2=e^{2x_3}\dsf x_1^2+e^{-2x_3}\dsf x_2^2+\dsf x_3^2
\end{equation}
where $(x_1,x_2,x_3)$ are canonical coordinates of $\rr^3$ (see for instance \cite{DHM09} and the references given there for more details). The space $\sol$ has a Lie group structure with respect to which the above metric is left-invariant. The group structure is given by the multiplication
$$(x_1,x_2,x_3)\cdot (y_1,y_2,y_3)=\Ngoac{x_1+e^{-x_3}y_1,x_2+e^{x_3}y_2,x_3+y_3}.$$
In this paper, we don't use the Lie group structure.
The change of coordinates
\begin{equation*}
x:=x_2,\qquad y:= e^{x_3},\qquad t:=x_1,
\end{equation*}
turns this model into $\sol=\{(x,y,t)\in\rr^3:y\ge 0 \}$ with the Riemannian  metric
\begin{equation}\label{equ2.2}
\dsf s^2=\dfrac{\dsf x^2+\dsf y^2}{y^2}+y^2\dsf t^2.
\end{equation}
In the present paper, the model used for the hyperbolic plane is the Poincar\'e half-plane, that is,
$$\hh^2=\{(x,y)\in\rr^2:y>0\}$$
endowed with the Riemannian  metric $\frac{\dsf x^2+\dsf y^2}{y^2}$. Hence $\sol$ has the form of a warped product $\sol=\hh^2\times_y\rr.$ 
From \eqref{equ2.2} we have
\begin{equation}
\td{\dhr_x}=\td{\dhr_y}=\frac{1}{y},\quad \td{\dhr_t}=y,\quad \vh{\dhr_x,\dhr_y}=\vh{\dhr_x,\dhr_t}=\vh{\dhr_y,\dhr_t}=0.
\end{equation}
Hence $\left\{y\dhr_x,y\dhr_y,\frac{1}{y}\dhr_t\right\}$ is an orthonormal frame of $\sol$.
Translations along the $t$-axis 
\begin{equation}\label{equ2.5}
T_h:\sol\to\sol,\qquad (x,y,t)\mapsto (x,y,t+h)
\end{equation}
are isometries. Therefore the vertical vector field $\dhr_t$ is a Killing vector field. Note that $\dhr_t$ is not unitary.

Let us denote by $\bn$ the Riemanian connexion of $\sol$ and by $\nabla$ the one in $\hh^2$.
By using Koszul's formula
\begin{align}\label{equKoszul}
2\vh{\bn_XY,Z}=& X\vh{Y,Z}+Y\vh{Z,X}-Z\vh{X,Y}\\
&-\vh{[X,Y],Z}-\vh{[Y,Z],X}+\vh{[Z,X],Y}\notag
\end{align}
for any vector field $X,Y,Z$ of $\sol$, we obtain
\begin{gather}
\overline\nabla_{\dhr_x}\dhr_x=\frac{1}{y}\dhr_y,\quad \overline\nabla_{\dhr_x}\dhr_y=\overline\nabla_{\dhr_y}\dhr_x=-\frac{1}{y}\dhr_x,\quad \overline\nabla_{\dhr_y}\dhr_y=-\frac{1}{y}\dhr_y,\\
\bn_{\dhr_t}\dhr_t=-y^3\dhr_y,\quad \bn_{\dhr_y}\dhr_t=\bn_{\dhr_t}\dhr_y=\frac{1}{y}\dhr_t,\\
\bn_{\dhr_x}\dhr_t=\bn_{\dhr_t}\dhr_x=0.
\end{gather}

Hence, the surfaces $\{t=\const\}$ and $\{x=\const\}$ are the totally geodesic surfaces in $\sol$ (Note that a totally geodesic submanifold $\Sigma\subset M$ is characterized by the fact that $\bn_XY$ is a tangent vector field of $\Sigma$ for all tangent vector fields $X,Y$ of $\Sigma$, where $\bn$ is the Riemannian connexion of $M$). The surfaces $\{y=\const\}$ are minimal, are not totally geodesic surfaces and are isometric to $\rr^2$.

\subsection{Euclidean geodesic}
Firstly, note that the vertical lines $\{p\}\times\rr\subset\sol$ with $p=(x,y)\in\hh^2$ aren't geodesics in $\sol$. Indeed, let $p=(x,y)$ be a point of $\hh^2$. A unit speed parametrization of $\gamma:=\{p\}\times\rr$ is
$\gamma:\rr\to\sol,\quad t\mapsto \left(x,y,\frac{t}{y}\right).$
One has $\gamma'=\frac{1}{y}\dhr_t$. Thus $\frac{\dsf }{\dsf t}\gamma'=\bn_{\frac{1}{y}\dhr_t}\left(\frac{1}{y}\dhr_t\right)=-y\dhr_y.$ Since $\frac{\dsf }{\dsf t}\gamma'\ne 0,$ $\{p\}\times\rr$ is not a geodesic in $\sol$.

\begin{md}\label{pro2.1}
Let $\gamma$ be a curve in $\hh^2$. Then the mean curvature vector of $\gamma\times\rr$ in $\sol$ is
$$\vec{H}_{\gamma\times\rr}=y^2\vec{\kappa}_\euc,$$
where $\vec{\kappa}_\euc$ is Euclidean mean curvature vector of $\gamma$ in $\hh^2$.
\end{md}
\begin{proof} We first compute $\vec{H}_{\gamma\times\rr}$.
Without loss of generality we can assume that $\gamma$ is a unit speed curve. So $\left\{\frac{1}{y}\dhr_t,\gamma'\right\}$ is an orthonormal frame of $\gamma\times\rr$. The mean curvature vector of $\gamma\times\rr$ is by definition
\begin{align}\label{HgR}
\vec{H}_{\gamma\times\rr}&=\Ngoac{\bn_{\frac{1}{y}\dhr_t}\left(\frac{1}{y}\dhr_t\right)+\bn_{\gamma'}\gamma'}^\bot\\
&=\Ngoac{-y\dhr_y+\bn_{\gamma'}\gamma'}^\bot\notag\\
&=-y\dhr_y^\bot+\vec{\kappa},\notag
\end{align}
where $\vec{\kappa}$ is the mean curvature vector of $\gamma$ in $\hh^2$.

We now compute the Euclidean mean curvature vector $\vec{\kappa}_\euc$ of $\gamma$ in $\hh^2$.
By Koszul's formula \eqref{equKoszul}
\begin{equation}\label{equEH}
\left(\nabla_\euc\right)_XY=\nabla_XY+\frac{1}{y}\Ngoac{(Xy)Y+(Yy)X-\vh{X,Y}\nabla y}
\end{equation}
where $\nabla_\euc$ is the Riemannian connexion of $\hh^2$ with respect to the Euclidean metric and $X,Y$ are tangent vector fields of $\hh^2$. Hence
\begin{equation}\label{equEH1}
\Ngoac{\left(\nabla_\euc\right)_XY}^\bot=\Ngoac{\nabla_XY}^\bot-\frac{1}{y}\vh{X,Y}\Ngoac{\nabla y}^\bot
\end{equation}
where $X,Y$ are tangent vector fields of $\gamma$. 
Since $\gamma$ is a unit speed curvature, $\td{\gamma'}=1$ and $\td{\frac{\gamma'}{y}}_\euc=1$. By \eqref{equEH1}
\begin{align*}
\vec{\kappa}_\euc&=\Ngoac{\Ngoac{\nabla_\euc}_{\frac{\gamma'}{y}}\frac{\gamma'}{y}}^\bot\\
&=\Ngoac{\nabla_{\frac{\gamma'}{y}}\frac{\gamma'}{y}}^\bot-\frac{1}{y}\vh{\frac{\gamma'}{y},\frac{\gamma'}{y}}\Ngoac{\nabla y}^\bot\\
&=\frac{1}{y^2}\vec{\kappa}-\frac{1}{y}\dhr_y^\bot
\end{align*}
Hence
$$y^2\vec{\kappa}_\euc=\vec{\kappa}-y\dhr_y^\bot.$$
Combining this equality with \eqref{HgR}, we complete the proof.
\end{proof}

Let us mention two important consequences of the proposition.

\begin{hq}\label{cor2.2}
Let $\gamma$ be a curve in $\hh^2$ and $\Omega$ be a domain in $\hh^2$ with $\dhr\Omega\in C^2$. Then 
\begin{enumerate}
\item $\gamma\times\rr$ is a minimal surface in $\sol$ if and only if $\gamma$ is an Euclidean geodesic in $\hh^2$. However, these Euclidean geodesics need not have constant speed parametrization.
\item $\Omega\times\rr$ is a mean convex set in $\sol$ if and only if $\Omega$ is an Euclidean mean convex in $\hh^2$.
\end{enumerate}
\end{hq}

\begin{md}\label{pro2.3}
Let $\gamma$ be a curve in $\hh^2$. Then the area calculated in $\sol$ of $\gamma\times[0,1]$ is
$$\Area(\gamma\times [0,1])=\ell_\euc(\gamma),$$
where $\ell_\euc(\gamma)$ is the Euclidean length of $\gamma$.
\end{md}
\begin{proof}
Let us first compute the area of $\gamma\times[0,1]$.
The surface $\gamma\times[0,1]$ in $\sol$ is defined by
$$\gamma\times[0,1]:[0,1]\times[0,1]\to\sol,\qquad (t_1,t_2)\mapsto(\gamma(t_1),t_2).$$
We have by definition
\begin{align*}
\Area(\gamma\times[0,1])&=\int_{[0,1]\times [0,1]}\td{(\gamma\times[0,1])_{t_1}\times (\gamma\times[0,1])_{t_2}}\,\dsf t_1\,\dsf t_2\\
&=\int_{0}^{1}\int_{0}^{1}\td{\gamma'(t_1)}y(\gamma(t_1))\,\dsf t_1\,\dsf t_2\\
&=\int_{0}^{1}\td{\gamma'(t_1)}y(\gamma(t_1))\,\dsf t_1=\int_{\gamma}y\,\dsf s.
\end{align*}
The Euclidean length of $\gamma$ is by definition
\begin{equation*}
\ell_\euc(\gamma)=\int_{\gamma}\dsf s_{\euc}=\int_{\gamma}y\,\dsf s.
\end{equation*}
Combining these equalities we conclude that
\begin{equation*}
\Area(\gamma\times[0,1])=\int_{\gamma}y\,\dsf s=\ell_\euc(\gamma).
\end{equation*}
This establishes the formula.
\end{proof}

The ideal boundary of $\hh^2$ is by definition
$$\dhr_\infty\hh^2=\{(x,y)\in\rr^2: y=0\}\cup\{\infty\}.$$
The point $\infty$ of $\dhr_\infty\hh^2$ is specified in our model of $\sol$  and we make the distinction with points in $\{y=0\}$.
\begin{dn}\label{aff}
A point $p\in\dhr_\infty\hh^2$ is called \textit{removable}  (resp. \textit{essential})
if $p\in \{(x,y)\in\rr^2: y=0\}$ (resp. $p=\infty$).
\end{dn}

\subsection{The minimal surface equations}

Let $\Omega$ be a domain in $\hh^2$ and $u$ be a $C^2$-function on $\Omega$. Using the previous model for $\hh^2$, we can consider the surface $\Gr(u)$ in $\sol$ parametrized by
$$(x,y)\mapsto (x,y,u(x,y)).$$
Such a surface is called the vertical Killing graph of $u$, it is transverse to the Killing vector field $\dhr_t$ and any integral curve of $\dhr_t$ intersect at most once the surface.
The upward unit normal to $\Gr(u)$ is given by
\begin{equation}\label{equ2.12}
N=N_u=\dfrac{-y\nabla u+\frac{1}{y}\dhr_t}{\sqrt{1+y^2\td{\nabla u}^2}},
\end{equation}
where $\nabla$ is the hyperbolic gradient operator and $\td{-}$ is the hyperbolic norm.
Indeed,  $\Gr(u)=\Phi^{-1}(0),$ where the function $\Phi:\sol\to\rr$ is defined by $\Phi(x,y,t)=t-u(x,y)$. So, $\bn\Phi$ is a normal vector field to $\Gr(u)$. Moreover, since $\bn t=\dfrac{1}{y^2}\dhr_t$ and $\vh{\bn u, \dhr_t}=0$, we have
$$\bn\Phi=\bn t-\bn u=\frac{1}{y^2}\dhr_t-\nabla u,\qquad \td{\bn\Phi}^2=\frac{1}{y^2}+\td{\nabla u}^2.$$
This establishes the formula \eqref{equ2.12}.

Denote
\begin{equation}\label{equ2.18}
W=W_u:=\sqrt{1+y^2\td{\nabla u}^2},\qquad X_u:=\dfrac{y\nabla u}{W}.
\end{equation}
It follows that
\begin{equation}
N=-X_u+\dfrac{1}{yW}\dhr_t.
\end{equation}
In the sequel, we will use this unit normal vector to compute the mean curvature of a Killing graph.

\begin{md}\label{equ-minimale} 
Let $\Omega$ be a domain in $\hh^2$ and $u$ be a $C^2$-function on $\Omega$. The mean curvature $H$ of the Killing graph of $u$ satisfies:
\begin{equation}\label{H1}
2yH=\Div\left(\dfrac{y^2\nabla u}{W}\right),
\end{equation}
with $\Div$ the  divergence operator in the hyperbolic metric, and after expanding all terms:
\begin{equation}\label{H1'}
2H=\dfrac{y^3}{W^3}\left(\left(1+y^4u_y^2\right)u_{xx}-2y^4u_xu_yu_{xy}+\left(1+y^4u_x^2\right)u_{yy}+2\frac{u_y}{y}\right).
\end{equation}
\end{md}
\begin{proof}
We extend the vector field $N$ to the whole $\Omega\times\rr$ by using the expression given in (\ref{equ2.12}). The mean curvature of the Killing graph $\Gr(u)$ of $u$ is then given by $2H=\Div_{\Gr(u)}(-N).$ 

Since $\dhr_t$ is a Killing vector field, we have
\begin{equation*}
2H=\Div_{\sol}(-N)=\Div_\sol (X_u)-\Div_\sol\left(\dfrac{1}{y W}\dhr_t\right).
\end{equation*}
Let us compute
$$\Div_{\sol}\left(\dfrac{1}{y W}\dhr_t\right)=\vh{\bn\dfrac{1}{y W},\dhr_t}+\dfrac{1}{y W}\Div_{\sol}\left(\dhr_t\right)=0+0=0,$$
$$\Div_\sol (X_u)=\Div(X_u)+\vh{\bn_{\frac{1}{y}\dhr_t}X_u,\frac{1}{y}\dhr_t}. $$
Moreover
$$\vh{\bn_{\frac{1}{y}\dhr_t}X_u,\frac{1}{y}\dhr_t}=\dfrac{1}{y^2}\vh{\bn_{\dhr_t}X_u,\dhr_t}=-\dfrac{1}{y^2}\vh{X_u,\bn_{\dhr_t}\dhr_t},$$
$$\bn_{\dhr_t}\dhr_t=-y^3\dhr_y=-y\nabla y.$$
Combining these equalities we deduce that
$$2H=\Div(X_u)+\dfrac{1}{y}\vh{X_u,\nabla y}.$$
It follows that
$$2yH=y\Div(X_u)+\vh{X_u,\nabla y}=\Div(yX_u)=\Div\left(\dfrac{y^2\nabla u}{W}\right).$$
This is the formula (\ref{H1}). Expanding (\ref{H1}) yields
\begin{align*}
2H&=\frac{1}{y}\Div\left(\dfrac{y^2\nabla u}{W}\right)=\frac{1}{y}\Div\left(\dfrac{y^4u_x}{W}\dhr_x+\dfrac{y^4u_y}{W}\dhr_y\right)\\
&=\frac{1}{y}\cdot y^2\left(\dfrac{\dhr}{\dhr x}\left(\dfrac{1}{y^2}\dfrac{y^4u_x}{W}\right) +\dfrac{\dhr}{\dhr_y}\left(\dfrac{1}{y^2}\dfrac{y^4u_y}{W}\right)\right)\\
&=\dfrac{y^3}{W^3}\left(\left(1+y^4u_y^2\right)u_{xx}-2y^4u_xu_yu_{xy}+\left(1+y^4u_x^2\right)u_{yy}+2\frac{u_y}{y}\right).
\end{align*}
This completes the proof.
\end{proof}

Thus the minimal surface equation for a function $u$ can be written
\begin{equation}\label{MSE}
\Mcal u:=\Div(yX_u)=0,
\end{equation}
\begin{equation}
\left(1+y^4u_y^2\right)u_{xx}-2y^4u_xu_yu_{xy}+\left(1+y^4u_x^2\right)u_{yy}+2\frac{u_y}{y}=0.
\end{equation}
A function $u\in C^2(\Omega)$ is said to be a \textit{minimal solution} on $\Omega$ mean that $u$  satisfies $\Mcal u=0$ on this domain.


\section{Maximum principle, Gradient estimate and Existence theorem}\label{sec3}

\subsection{Maximum principle}
A basic tool for obtaining the results of this work is the maximum principle for differences of minimal solutions.

Firstly, by applying the proof of
\cite[Theorem 10.1]{GT01} we have
\begin{md}[Maximum principle]\label{PM}
Let $u_1,u_2$ be two $C^2$-functions on a domain $\Omega\subset\hh^2$. Suppose $u_1$ and $u_2$ satisfy $\Mcal u_1\ge \Mcal u_2$. Then $u_2-u_1 $ cannot have an interior minimum unless $u_2-u_1$ is a constant.
\end{md}

It follows from this proposition that  \textit{if $u_1,u_2$ are  functions of class $C^2$ on a bounded domain $\Omega\subset\hh^2$ such that $\Mcal u_1\ge \Mcal u_2$, and $\liminf\limits(u_2-u_1)\ge 0$ for any approach to the boundary $\dhr\Omega$ of $\Omega$, then we have $u_2 \ge u_1$ in $\Omega$}.

Indeed, assume the contrary that $\{x\in\Omega: u_2(x)<u_1(x)\}$ is not empty. Since $\liminf\limits(u_2-u_1)\ge 0$ for any approach to the boundary $\dhr\Omega$ and $\Omega$ is bounded, $u_2-u_1$ has an interior minimum in $\Omega$. By Proposition \ref{PM}, $u_2-u_1$ is constant, a contradiction.

The following result (Theorem \ref{PMG}) is a remarkable strengthening of this situation.

In what follows, for a subset $\Omega$ of $\hh^2$, we will denote by $\dhr_\infty\Omega$ the boundary of $\Omega$ in $\hh^2\cup\dhr_\infty\hh^2$.

\begin{dn}
A domain $\Omega\subset\hh^2$ is called \textit{admissible} if its  boundary $\dhr_\infty\Omega$ is composed of a finite number of open Euclidean convex arcs $C_i$ in $\hh^2$ together with their endpoints. The endpoints of the arcs $C_i$ are called vertices of $\Omega$ and those in $\dhr_\infty\hh^2$ are called ideal vertices of $\Omega$. Assume in addition that, the ideal vertices of this domain are  removable points at infinity (see Figure \ref{F1}).
\end{dn}


\begin{figure}[!h]
\centering

\definecolor{xdxdff}{rgb}{0.49,0.49,1}
\definecolor{qqqqff}{rgb}{0,0,1}
\begin{tikzpicture}[line cap=round,line join=round,>=triangle 45,x=1.0cm,y=1.0cm]
\clip(-2.62,-1.2) rectangle (7.38,5.56);
\draw (1.36,-0.48)-- (-0.64,1.52);
\draw [shift={(4.02,1.25)}] plot[domain=-1.37:1.4,variable=\t]({1*1.76*cos(\t r)+0*1.76*sin(\t r)},{0*1.76*cos(\t r)+1*1.76*sin(\t r)});
\draw [shift={(2.74,2.69)}] plot[domain=0.18:2.6,variable=\t]({1*1.61*cos(\t r)+0*1.61*sin(\t r)},{0*1.61*cos(\t r)+1*1.61*sin(\t r)});
\draw [shift={(0.9,1.52)}] plot[domain=4.94:6.06,variable=\t]({1*2.05*cos(\t r)+0*2.05*sin(\t r)},{0*2.05*cos(\t r)+1*2.05*sin(\t r)});
\draw [shift={(0.13,2.69)}] plot[domain=-0.35:0.59,variable=\t]({1*1.48*cos(\t r)+0*1.48*sin(\t r)},{0*1.48*cos(\t r)+1*1.48*sin(\t r)});
\draw (1.52,2.18)-- (2.9,3.04);
\draw [shift={(4.99,4.05)}] plot[domain=4.34:4.8,variable=\t]({1*2.52*cos(\t r)+0*2.52*sin(\t r)},{0*2.52*cos(\t r)+1*2.52*sin(\t r)});
\draw [shift={(5.66,0.64)}] plot[domain=2.02:3.08,variable=\t]({1*1*cos(\t r)+0*1*sin(\t r)},{0*1*cos(\t r)+1*1*sin(\t r)});
\draw (1.36,-0.48)-- (2.18,1);
\draw [shift={(-0.33,0.24)}] plot[domain=-0.4:0.66,variable=\t]({1*1.83*cos(\t r)+0*1.83*sin(\t r)},{0*1.83*cos(\t r)+1*1.83*sin(\t r)});
\draw [->,dash pattern=on 2pt off 2pt] (-2.64,-0.48) -- (7.36,-0.48);
\draw (5.78,5.28) node[anchor=north west] {$\mathbb{H}^2$};
\draw (2.84,2.3) node[anchor=north west] {$\Omega$};
\draw (-0.1,0.7) node[anchor=north west] {$C_1$};
\draw (-0.76,3.6) node[anchor=north west] {$C_2$};
\draw (1.05,3.2) node[anchor=north west] {$C_3$};
\draw (2,2.68) node[anchor=north west] {$C_4$};
\draw (1.74,3.5) node[anchor=north west] {$C_5$};
\draw (3.42,4.7) node[anchor=north west] {$C_6$};
\draw (5.86,1.68) node[anchor=north west] {$C_7$};
\draw (3.22,0.5) node[anchor=north west] {$C_8$};
\draw (2.22,0.26) node[anchor=north west] {$C_9$};
\draw (1.74,0.66) node[anchor=north west] {$C_{10}$};
\draw (0.8,0.9) node[anchor=north west] {$C_{11}$};
\draw (1.32,1.5) node[anchor=north west] {$C_{12}$};
\draw (4.38,2.02) node[anchor=north west] {$C_{13}$};
\draw (3.8,1.5) node[anchor=north west] {$C_{14}$};
\draw (4.8,1.4) node[anchor=north west] {$C_{15}$};
\draw [shift={(1.04,1.84)}] plot[domain=1.38:3.33,variable=\t]({1*1.71*cos(\t r)+0*1.71*sin(\t r)},{0*1.71*cos(\t r)+1*1.71*sin(\t r)});
\draw [shift={(1.92,1.98)}] plot[domain=3.8:4.97,variable=\t]({1*1.02*cos(\t r)+0*1.02*sin(\t r)},{0*1.02*cos(\t r)+1*1.02*sin(\t r)});
\draw [shift={(4.71,1.31)}] plot[domain=3.28:4.52,variable=\t]({1*1.83*cos(\t r)+0*1.83*sin(\t r)},{0*1.83*cos(\t r)+1*1.83*sin(\t r)});
\draw [shift={(3.25,0.55)}] plot[domain=0.11:0.94,variable=\t]({1*1.42*cos(\t r)+0*1.42*sin(\t r)},{0*1.42*cos(\t r)+1*1.42*sin(\t r)});
\draw [shift={(2.5,4.47)}] plot[domain=3.83:4.99,variable=\t]({1*1.48*cos(\t r)+0*1.48*sin(\t r)},{0*1.48*cos(\t r)+1*1.48*sin(\t r)});
\draw [->,dash pattern=on 2pt off 2pt] (-1.48,-1.16) -- (-1.44,5.38);
\draw (7.04,0.1) node[anchor=north west] {$x$};
\draw (-1.26,5.5) node[anchor=north west] {$y$};
\draw (-1.4,-0.54) node[anchor=north west] {$O$};
\begin{scriptsize}
\draw [fill=qqqqff] (1.36,-0.48) circle (1.5pt);
\draw [fill=qqqqff] (1.36,3.52) circle (1.5pt);
\draw [fill=qqqqff] (4.36,-0.48) circle (1.5pt);
\draw [fill=qqqqff] (4.32,2.98) circle (1.5pt);
\draw [fill=qqqqff] (2.9,1.06) circle (1.5pt);
\draw [fill=qqqqff] (1.52,2.18) circle (1.5pt);
\draw [fill=qqqqff] (2.9,3.04) circle (1.5pt);
\draw [fill=qqqqff] (4.08,1.7) circle (1.5pt);
\draw [fill=qqqqff] (5.22,1.54) circle (1.5pt);
\draw [fill=qqqqff] (4.66,0.7) circle (1.5pt);
\draw [fill=qqqqff] (2.18,1) circle (1.5pt);
\draw [fill=qqqqff] (1.12,1.36) circle (1.5pt);
\draw [fill=qqqqff] (-0.64,1.52) circle (1.5pt);
\draw[color=qqqqff] (-0.4,1.86) node {};
\draw [fill=qqqqff] (1.36,3.52) circle (1.5pt);
\draw[color=qqqqff] (1.6,3.86) node {};
\draw [fill=qqqqff] (1.12,1.36) circle (1.5pt);
\draw[color=qqqqff] (1.36,1.7) node {};
\draw [fill=qqqqff] (2.18,1) circle (1.5pt);
\draw[color=qqqqff] (2.42,1.34) node {};
\draw [fill=qqqqff] (4.36,-0.48) circle (1.5pt);
\draw[color=qqqqff] (4.6,-0.14) node {};
\draw [fill=qqqqff] (2.9,1.06) circle (1.5pt);
\draw[color=qqqqff] (3.14,1.4) node {};
\draw [fill=qqqqff] (4.08,1.7) circle (1.5pt);
\draw[color=qqqqff] (4.32,2.04) node {};
\draw [fill=qqqqff] (4.66,0.7) circle (1.5pt);
\draw[color=qqqqff] (4.9,1.04) node {};
\draw [fill=qqqqff] (1.36,3.52) circle (1.5pt);
\draw[color=qqqqff] (1.6,3.86) node {};
\draw [fill=qqqqff] (2.9,3.04) circle (1.5pt);
\draw[color=qqqqff] (3.14,3.38) node {};
\draw [fill=xdxdff] (-1.48,-0.48) circle (1.5pt);
\end{scriptsize}
\end{tikzpicture}

\caption{An example of admissible domain}\label{F1}
\end{figure}
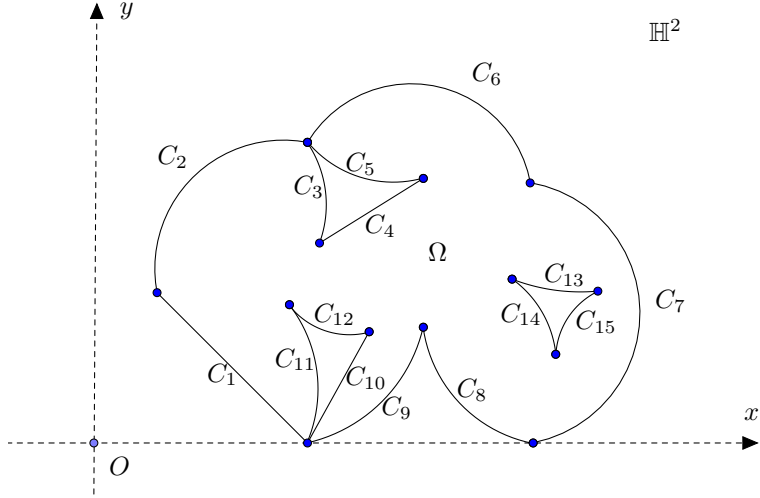

\begin{md}[General maximum principle]\label{PMG}
Let $\Omega\subset\hh^2$ be a admissible domain. Let $u_1,u_2$ be two minimal solutions on $\Omega$. Suppose that $\limsup(u_1-u_2)\le  0$ for any approach to the boundary of $\Omega$ exception of its vertices. Then $u_1\le  u_2.$
\end{md}

We should remark that this result is similar to the general maximum principle stated by Spruck \cite[General Maximum Principle, page 3]{Spr72}  (resp. Hauswirth-Rosenberg-Spruck \cite[Theorem 2.2]{HRS09})  for constant mean curvature surfaces in $\rr^2\times\rr$ (resp. in $\hh^2\times\rr$ and $\mathbb{S}^2\times\rr$) in the case of the bounded domain $\Omega$ and by Collin-Rosenberg \cite[Theorem 2]{CR10}  for minimal surfaces in $\hh^2\times\rr$ in the case of the unbounded domain $\Omega$.

\begin{proof} 
Assume the contrary, that the set $\{u_1>u_2\}$ is non-empty.

Let $N$ and $\varepsilon$ be positive constants, with $N$ large and $\varepsilon$ small. Define
$$\varphi=\left[u_1-u_2-\varepsilon\right]_{0}^{N}=\begin{cases}
0& \text{ if }u_1-u_2-\varepsilon\le 0,\\
u_1-u_2-\varepsilon&\text{ if } 0<u_1-u_2-\varepsilon <N,\\
N& \text{ if } u_1-u_2-\varepsilon\ge N.
\end{cases}$$

Then $\varphi$ is a continuous piecewise differentiable function in $\Omega$ satisfying $0\le \varphi <N$. Moreover $\nabla\varphi=\nabla u_1-\nabla u_2$ in the set where $\varepsilon<u_1-u_2<N+\varepsilon $, and $\nabla\varphi=0$ almost every where in the complement of this set. For each ideal vertex $p$ of $\Omega$, we consider a sequence of nested ideal geodesics $H_{p,n}$ converging to $p$. By nested we mean that if $\Hcal_{p,n}$ is the component of $\hh^2\setminus H_{p,n}$ containing $p$ on its ideal boundary, then $\Hcal_{p,n+1}\subset\Hcal_{p,n}$. Assume $\overline{\Hcal}_{p_1,n}\cap\overline{\Hcal}_{p_2,n}=\emptyset$ for every different ideal vertices $p_1,p_2$ of $\Omega$.
Define
\[\Omega_n=\Omega\setminus\left(\bigcup_{p\in E_1}\overline\Dbb_{\frac{1}{n}}(p)\cup \bigcup_{p\in E_2}\overline{\Hcal}_{p,n}\right),\quad \Gamma_1=\dhr\Omega_n\cap\dhr\Omega,\quad \Gamma_2=\dhr\Omega_n\setminus\Gamma_1,\]
where $E_1$ (resp. $E_2$) is the set of vertices in $\hh^2$ (resp. vertices at $\dhr_\infty\hh^2$) of $\Omega$ (see Figure \ref{F2}).

It follows from definition that
\begin{equation}\label{tc1}
\varphi=0 \text{ on a neighborhood of $\Gamma_1$},\qquad \ell_\euc(\Gamma_2)\to 0\text{ as $n\to\infty$}.
\end{equation}


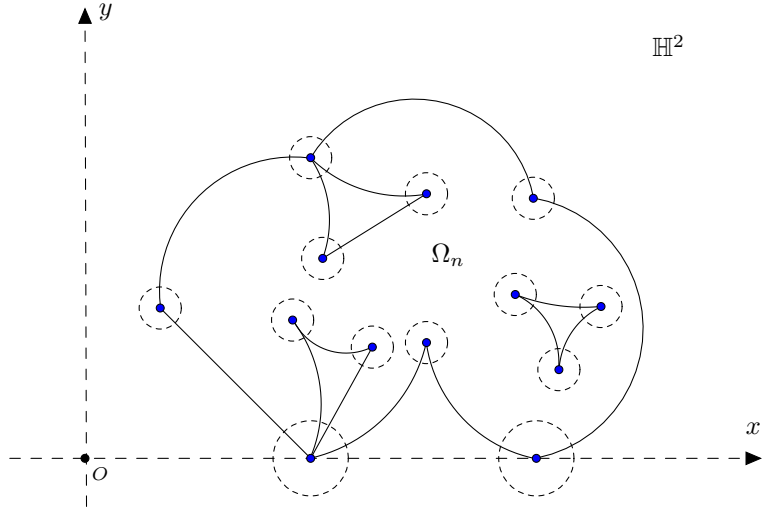
\begin{figure}[!h]
\centering
\definecolor{qqqqff}{rgb}{0,0,1}
\begin{tikzpicture}[line cap=round,line join=round,>=triangle 45,x=1.0cm,y=1.0cm]
\clip(-3.32,-1.74) rectangle (7.28,5.26);
\draw (1,-1)-- (-1,1);
\draw [shift={(3.66,0.73)}] plot[domain=-1.37:1.4,variable=\t]({1*1.76*cos(\t r)+0*1.76*sin(\t r)},{0*1.76*cos(\t r)+1*1.76*sin(\t r)});
\draw [shift={(2.38,2.17)}] plot[domain=0.18:2.6,variable=\t]({1*1.61*cos(\t r)+0*1.61*sin(\t r)},{0*1.61*cos(\t r)+1*1.61*sin(\t r)});
\draw [shift={(0.54,1)}] plot[domain=4.94:6.06,variable=\t]({1*2.05*cos(\t r)+0*2.05*sin(\t r)},{0*2.05*cos(\t r)+1*2.05*sin(\t r)});
\draw [shift={(-0.23,2.17)}] plot[domain=-0.35:0.59,variable=\t]({1*1.48*cos(\t r)+0*1.48*sin(\t r)},{0*1.48*cos(\t r)+1*1.48*sin(\t r)});
\draw (1.16,1.66)-- (2.54,2.52);
\draw [shift={(4.63,3.53)}] plot[domain=4.34:4.8,variable=\t]({1*2.52*cos(\t r)+0*2.52*sin(\t r)},{0*2.52*cos(\t r)+1*2.52*sin(\t r)});
\draw [shift={(5.3,0.12)}] plot[domain=2.02:3.08,variable=\t]({1*1*cos(\t r)+0*1*sin(\t r)},{0*1*cos(\t r)+1*1*sin(\t r)});
\draw (1,-1)-- (1.82,0.48);
\draw [shift={(-0.69,-0.28)}] plot[domain=-0.4:0.66,variable=\t]({1*1.83*cos(\t r)+0*1.83*sin(\t r)},{0*1.83*cos(\t r)+1*1.83*sin(\t r)});
\draw [->,dash pattern=on 4pt off 4pt] (-3,-1) -- (7,-1);
\draw (5.42,4.76) node[anchor=north west] {$\mathbb{H}^2$};
\draw (2.48,1.98) node[anchor=north west] {$\Omega_n$};
\draw [dash pattern=on 2pt off 2pt] (1,3) circle (0.28cm);
\draw [dash pattern=on 2pt off 2pt] (2.54,2.52) circle (0.28cm);
\draw [dash pattern=on 2pt off 2pt] (1.16,1.66) circle (0.28cm);
\draw [dash pattern=on 2pt off 2pt] (3.96,2.46) circle (0.28cm);
\draw [dash pattern=on 2pt off 2pt] (4.86,1.02) circle (0.28cm);
\draw [dash pattern=on 2pt off 2pt] (3.72,1.18) circle (0.28cm);
\draw [dash pattern=on 2pt off 2pt] (4.3,0.18) circle (0.28cm);
\draw [dash pattern=on 2pt off 2pt] (2.54,0.54) circle (0.28cm);
\draw [dash pattern=on 2pt off 2pt] (1.82,0.48) circle (0.28cm);
\draw [dash pattern=on 2pt off 2pt] (0.76,0.84) circle (0.28cm);
\draw [dash pattern=on 2pt off 2pt] (-1,1) circle (0.28cm);
\draw [dash pattern=on 2pt off 2pt] (1,-1) circle (0.5cm);
\draw [dash pattern=on 2pt off 2pt] (4,-1) circle (0.5cm);
\draw [shift={(0.77,1.23)}] plot[domain=1.44:3.27,variable=\t]({1*1.78*cos(\t r)+0*1.78*sin(\t r)},{0*1.78*cos(\t r)+1*1.78*sin(\t r)});
\draw [shift={(1.47,1.18)}] plot[domain=3.59:5.18,variable=\t]({1*0.78*cos(\t r)+0*0.78*sin(\t r)},{0*0.78*cos(\t r)+1*0.78*sin(\t r)});
\draw [shift={(4.35,0.79)}] plot[domain=3.28:4.52,variable=\t]({1*1.83*cos(\t r)+0*1.83*sin(\t r)},{0*1.83*cos(\t r)+1*1.83*sin(\t r)});
\draw [shift={(2.24,4.26)}] plot[domain=3.94:4.88,variable=\t]({1*1.77*cos(\t r)+0*1.77*sin(\t r)},{0*1.77*cos(\t r)+1*1.77*sin(\t r)});
\draw [shift={(3.31,0.28)}] plot[domain=-0.1:1.15,variable=\t]({1*0.99*cos(\t r)+0*0.99*sin(\t r)},{0*0.99*cos(\t r)+1*0.99*sin(\t r)});
\draw [->,dash pattern=on 4pt off 4pt] (-1.98,-1.64) -- (-2,5);
\draw (-1.94,5.18) node[anchor=north west] {$y$};
\draw (6.66,-0.4) node[anchor=north west] {$x$};
\begin{scriptsize}
\draw [fill=qqqqff] (1,-1) circle (1.5pt);
\draw [fill=qqqqff] (-1,1) circle (1.5pt);
\draw [fill=qqqqff] (1,3) circle (1.5pt);
\draw [fill=qqqqff] (4,-1) circle (1.5pt);
\draw [fill=qqqqff] (3.96,2.46) circle (1.5pt);
\draw [fill=qqqqff] (2.54,0.54) circle (1.5pt);
\draw [fill=qqqqff] (1.16,1.66) circle (1.5pt);
\draw [fill=qqqqff] (2.54,2.52) circle (1.5pt);
\draw [fill=qqqqff] (3.72,1.18) circle (1.5pt);
\draw [fill=qqqqff] (4.86,1.02) circle (1.5pt);
\draw [fill=qqqqff] (4.3,0.18) circle (1.5pt);
\draw [fill=qqqqff] (1.82,0.48) circle (1.5pt);
\draw [fill=qqqqff] (0.76,0.84) circle (1.5pt);
\draw [fill=black] (-2,-1) circle (1.5pt);
\draw[color=black] (-1.8,-1.2) node {$O$};
\end{scriptsize}
\end{tikzpicture}
\caption{The domain $\Omega_n$}\label{F2}
\end{figure}


\noindent Define
$$J_n=\int_{\dhr\Omega_n}\varphi y\vh{X_{u_1}-X_{u_2},\nu} \,\dsf s$$
where $\nu$ is the exterior normal to $\dhr\Omega_n$, $W_{u_i}=\sqrt{1+y^2\td{\nabla u_i}^2}$ and $X_{u_i}=\dfrac{y\nabla u_i}{W_{u_i}}, i=1,2.$

\begin{kd}
\begin{itemize}
\item[\rm (i)] $J_n\ge 0$ with equality if and only if  $\nabla u_1=\nabla u_2$ on the set $\{x\in\Omega_n:\varepsilon<u_1-u_2<N\}$.
\item[\rm (ii)] $J_n$ is increasing as $n\to\infty$.
\end{itemize}
\end{kd}
\begin{proof}
By Divergence theorem, we have
\begin{align*}
J_n&=\int_{\Omega_n}\Div\left( \varphi y(X_{u_1}-X_{u_2})\right)\,\dsf\Acal\\
&=\int_{\Omega_n}\langle y\nabla\varphi,X_{u_1}-X_{u_2}\rangle\,\dsf\Acal+\int_{\Omega_n}\varphi\Div(yX_{u_1}-yX_{u_2})\,\dsf\Acal.
\end{align*}
By our assumptions,
$$\varphi\Div(yX_{u_1}-yX_{u_2})=\varphi(\Mcal u_1-\Mcal u_2)= 0.$$
Moreover, by  formula (\ref{equ3.3}) of Lemma \ref{lem cle}
$$\langle y\nabla\varphi,X_{u_1}-X_{u_2}\rangle=\vh{y\nabla u_1-y\nabla u_2,\frac{y\nabla u_1}{W_{u_1}}-\frac{y\nabla u_2}{W_{u_2}}}\ge 0.$$
and equality if and only if $y\nabla u_1=y\nabla u_2$. \\
Then $$J_n=\int_{\Omega_n}\Div\left( \varphi y(X_{u_1}-X_{u_2})\right)\,\dsf\Acal=\int_{\Omega_n}\langle y\nabla\varphi,X_{u_1}-X_{u_2}\rangle\,\dsf\Acal  \ge 0 $$ 
 and $J_n=0 $ if $\nabla u_1=\nabla u_2$. Since $\Omega_n$ is an increasing domain, $J_n$ is increasing. 
This proves the assertion.
\renewcommand{\qedsymbol}{$\Diamond$}
\end{proof}

\begin{kd}
$J_n=o(1)$ as $n\to\infty$.
\end{kd}
\begin{proof}
We have
$$J_n=\int_{\Gamma_1}\varphi y\vh{X_{u_1}-X_{u_2},\nu}\,\dsf s+\int_{\Gamma_2}\varphi y\vh{X_{u_1}-X_{u_2},\nu}\,\dsf s.$$
By (\ref{tc1}),  and $\td{X_{u^i}}\le 1,i=1,2;$ $0\le \varphi\le N$, we have
$$\int_{\Gamma_1}\varphi y\vh{X_{u_1}-X_{u_2},\nu}\, \dsf s=0$$
and
\begin{align*}
\left|\int_{\Gamma_2}\varphi y\vh{X_{u_1}-X_{u_2},\nu}\, \dsf s\right|&=\left|\int_{\Gamma_2}\varphi \vh{X_{u_1}-X_{u_2},\nu}\,\dsf s_{\euc} \right|\\
&\le 2N\ell_\euc(\Gamma_2)=o(1) \quad\text{ as $n\to \infty$}.
\end{align*}
Assertion is then proved.
\renewcommand{\qedsymbol}{$\Diamond$}
\end{proof}

It follows from the previous assertions that $\nabla u_1=\nabla u_2$ on the set $\{\varepsilon<u_1-u_2<N \}$. Since $\varepsilon>0$ and $N$ are arbitrary, $\nabla u_1=\nabla u_2$ whenever $u_1> u_2$. So $u_1=u_2+c, (c>0)$ in any nontrivial component of the set $\{u_1>u_2\}$. Then the maximum principle (Theorem \ref{PM}) ensures  $u_1=u_2+c$ in $\Omega$ and by assumptions of the theorem, the constant must be nonpositive, a contradiction.
\end{proof}

\begin{bd}\label{lem cle}
Let $v_1,v_2$ be two vectors in a finite dimensional Euclidean space. Then
\begin{equation}
(v_1-v_2)\left(\dfrac{v_1}{W_1}-\dfrac{v_2}{W_2}\right)=\dfrac{W_1+W_2}{2}\left(\left\|\dfrac{v_1}{W_1}-\dfrac{v_2}{W_2} \right\|^2+\left(\dfrac{1}{W_1}-\dfrac{1}{W_2}\right)^2\right).
\end{equation}
where $W_i=\sqrt{1+\td{v_i}^2}.$
In particular,
\begin{equation}\label{equ3.3}
(v_1-v_2)\left(\dfrac{v_1}{W_1}-\dfrac{v_2}{W_2}\right)\ge \left\|\dfrac{v_1}{W_1}-\dfrac{v_2}{W_2} \right\|^2,\quad (v_1-v_2)\left(\dfrac{v_1}{W_1}-\dfrac{v_2}{W_2}\right)\ge 0
\end{equation}
with equality at a point if and only if $v_1=v_2$.
\end{bd}

\begin{proof} Let us compute
\begin{align*}
(v_1-v_2)(\dfrac{v_1}{W_1}-\dfrac{v_2}{W_2})=& \dfrac{\|v_1\|^2}{W_1}+ \dfrac{\|v_2\|^2}{W_2}- \langle v_1,v_2\rangle \left(\dfrac{1}{W_1}+\dfrac{1}{W_2}\right)\\
=&W_1-\dfrac{1}{W_1} + W_2-\dfrac{1}{W_2}- \langle v_1,v_2\rangle \left(\dfrac{1}{W_1}+\dfrac{1}{W_2}\right)\\
=&(W_1+W_2)\left(1- \dfrac{\langle v_1,v_2\rangle}{W_1W_2} -\dfrac{1}{W_1W_2}\right)\\
=&(W_1+W_2)\left(\dfrac{1}{2}\left\|\dfrac{v_1}{W_1}-\dfrac{v_2}{W_2} \right\|^2+ \dfrac{1}{2W_1^2}\right.\\
&\left.+\dfrac{1}{2W_2^2}-\dfrac{1}{W_1W_2}\right)\\
=&\dfrac{W_1+W_2}{2}\left(\left\|\dfrac{v_1}{W_1}-\dfrac{v_2}{W_2} \right\|^2+\left(\dfrac{1}{W_1}-\dfrac{1}{W_2}\right)^2\right).
\end{align*}
This proves the lemma.
\end{proof}

\subsection{Gradient estimate}
An important result concerning minimal solutions  is a gradient estimate.
\begin{dl}[Interior gradient estimate]\label{EGI}
Let $u$ be a nonnegative minimal solution on $\Omega=B_R(p)\subset\hh^2$ . Then there exists a constant $C$ that depends only on $p, R$ such that
\begin{equation}
\td{\nabla u(p)}\le f\left(\dfrac{u(p)}{R}\right),\qquad f(t)=e^{C(1+t^2)}.
\end{equation}
\end{dl}

The proof of this result is similar to the one of the gradient estimate proved by Spruck \cite[Theorem 1.1]{Spr07} and Mazet \cite[Proposition 16]{Maz13}.

Before beginning the proof, let us make some preliminary computation.

\begin{bd}
Let $u$ be a minimal solution on a domain $\Omega\subset\hh^2$. Denote by $\Sigma$ the graph of $u$.
Then
\begin{equation}
\nabla_\Sigma u=\dfrac{1}{y^2}\dhr_t^\top,\qquad \td{\nabla_\Sigma u}^2=\dfrac{1}{y^2}\left(1-\dfrac{1}{W^2}\right) \quad\text{and}\qquad
\Delta_\Sigma u=\dfrac{2\vh{\dhr_y,N}}{W}
\end{equation}
where the subscript $._\Sigma$ signifies that we compute the object in the Riemannian  metric of the surface $\Sigma$.
\end{bd}
\begin{proof} We have
$$\nabla_\Sigma u=\nabla_\Sigma t = (\overline\nabla t)^\top= \dfrac{1}{y^2}\dhr_t^\top,$$
$$\dhr_t ^\top = \dhr_t- \vh{\dhr_t, N}N= \dhr_t- \dfrac{y}{W}N.$$
It follows that
\begin{align*}
\td{\nabla_\Sigma u}^2&=\dfrac{1}{y^4}\td{\dhr_t ^\top}^2 =\dfrac{1}{y^4} \left(\td{\dhr_t}^2-\dfrac{y^2}{W^2}\td{N}^2\right)\\
&= \dfrac{1}{y^4}\left (y^2-\dfrac{y^2}{W^2}\right )=\dfrac{1}{y^2}\left(1-\dfrac{1}{W^2}\right).
\end{align*}
We continue to compute $\Delta_\Sigma u$
$$\Delta_\Sigma u=\Div_\Sigma\nabla_\Sigma t=\Div_\Sigma(\bn t)^\top=\Div_\Sigma(\bn t) +\vh {2\vec{H}, \overline\nabla t}= \Div_\Sigma(\bn t).$$
Moreover
$$\Div_\Sigma(\bn t)=\Div_\Sigma\left(\dfrac{1}{ y^2}\dhr_t\right)=\vh{\nabla_\Sigma \frac{1}{ y^2},\dhr_t}+\dfrac{1}{ y^2}\Div_\Sigma(\dhr_t)=-\dfrac{2}{ y^3}\vh{\nabla_\Sigma y,\dhr_t}.$$
By using $\nabla_\Sigma y=\nabla y-\vh{\nabla y,N}N$, $\vh{\dhr_t, N}=\dfrac{y}{W}$, we obtain
$$\Div_\Sigma(\bn t)=\dfrac{2}{y^3}\vh{\nabla y, N}\vh{\dhr_t, N}=\dfrac{2}{W y^2}\vh{\nabla y,N}. $$
We conclude that
$$\Delta_\Sigma u=\dfrac{2}{W y^2}\vh{\nabla y,N}=\dfrac{2\vh{\dhr_y,N}}{W}.$$
This completes the proof of the lemma.
\end{proof}
Since $\dhr_t$ is a Killing vector field and $\frac{y}{W}=\vh{\dhr_t,N}$, then
\begin{equation}
\Delta_\Sigma \dfrac{y}{W}=-\left(\td{A}^2+\Ric(N,N)\right)\dfrac{y}{W}.
\end{equation}

\begin{bd}Let $u$ be a minimal solution on a domain $\Omega\subset\hh^2$. Denote by $\Sigma $ the graph of $u$ on the domain $\Omega$.

 For each function $\varphi:\Omega\to\rr$ then
\begin{equation}
\Delta_\Sigma\varphi=\Delta\varphi-\dfrac{ y^2}{W^2}\vh{\nabla_{\nabla u}\nabla\varphi,\nabla u}+\dfrac{1}{ y}\left(1-\dfrac{1}{W^2}\right)\vh{\nabla\varphi,\nabla y}.
\end{equation}
\end{bd}
\begin{proof} We have
\begin{align*}
\Delta_\Sigma\varphi&=\Div_\Sigma\nabla_\Sigma\varphi=\Div_\Sigma\nabla\varphi +2\vh{\nabla \varphi, \vec{H}}=\Div_\Sigma\nabla\varphi\\ &=\Div_{\sol}\nabla\varphi-\vh{\overline \nabla_N\nabla\varphi,N}.
\end{align*}
Let us evaluate the terms in the right-hand side
\begin{align*}
\Div_{\sol}\nabla\varphi&=\Div\nabla\varphi+\vh{\overline\nabla_{\frac{1}{ y}\dhr_t}\nabla\varphi,\frac{1}{ y}\dhr_t}=\Delta\varphi+\dfrac{1}{ y^2}\vh{\overline\nabla_{\dhr_t}\nabla\varphi,\dhr_t}\\
&=\Delta\varphi-\dfrac{1}{ y^2} \vh{\overline\nabla_{\dhr_t}\dhr_t, \nabla\varphi}=\Delta\varphi+\dfrac{1}{ y}\vh{\nabla\varphi,\nabla y}.
\end{align*}
Since $N=-\frac{ y\nabla u}{W}+\frac{\dhr_t}{ y W}$, 
\begin{align*}
\vh{\overline\nabla_N\nabla\varphi,N}&=\vh{\overline\nabla_{-\frac{ y\nabla u}{W}}\nabla\varphi,-\frac{ y\nabla u}{W}}+\vh{\overline\nabla_{\frac{\dhr_t}{ y W}}\nabla\varphi,\frac{\dhr_t}{ y W}}\\
&=\dfrac{ y^2}{W^2}\vh{\nabla_{\nabla u}\nabla\varphi,\nabla u}+\dfrac{1}{ y^2W^2}\vh{\overline\nabla_{\dhr_t}\nabla\varphi,\dhr_t}\\
&=\dfrac{ y^2}{W^2}\vh{\nabla_{\nabla u}\nabla\varphi,\nabla u}+\dfrac{1}{ y W^2}\vh{\nabla\varphi,\nabla y}.
\end{align*}
It follows that
$$\Delta_\Sigma\varphi=\Delta\varphi-\dfrac{ y^2}{W^2}\vh{\nabla_{\nabla u}\nabla\varphi,\nabla u}+\dfrac{1}{ y}\left(1-\dfrac{1}{W^2}\right)\vh{\nabla\varphi,\nabla y},$$
which completes the proof.
\end{proof}

Let us mention an important consequence of the lemma.
\begin{hq}\label{dist}
Let $\Omega\subset\hh^2$ be a bounded domain, let $p$ be a point of $\Omega$. Denote by  $d=d_{\hh^2}(\cdot,p)$ the hyperbolic distance from a point in $\Omega$ to $p$. Let $u$ be a minimal solution on $\Omega$. There exists a constant $C=C(\Omega, y)$ ($C$ doesn't depend on the point $p$ and the function $u$) such that
\begin{equation}
\sup_\Omega |\Delta_\Sigma d^2|\le C.
\end{equation}
\end{hq}


Using the above computations, we are ready to write the proof.
\begin{proof}[Proof of Theorem $\ref{EGI}$]
Let us denote $\nu:=\dfrac{ y}{W}=\vh{\dhr_t,N}$. By definition, $\dhr_t=\dhr_t^\top+\nu N$.

We define an operator on $\Sigma$
\begin{equation}
Lf:=\Delta_\Sigma f-2\nu\vh{\nabla_\Sigma\frac{1}{\nu},\nabla_\Sigma f}.
\end{equation}
We remark that the maximum principle is true for $L$. We have
\begin{align*}
\Delta_\Sigma\dfrac{1}{\nu}&=-\dfrac{1}{\nu^2}\Delta_\Sigma\nu+\dfrac{2}{\nu^3}\td{\nabla_\Sigma \nu}^2\\
&=-\dfrac{1}{\nu^2}\left(-\left(\Ric(N,N)+\td{A}^2\right)\nu\right)+\dfrac{2}{\nu^3}\td{-\nu^2\nabla_\Sigma \dfrac{1}{\nu}}^2\\
&=\left(\Ric(N,N)+\td{A}^2\right)\dfrac{1}{\nu}+2\nu\td{\nabla_\Sigma \dfrac{1}{\nu}}^2.
\end{align*}
Therefore
$$L\dfrac{1}{\nu}=\Delta_\Sigma \dfrac{1}{\nu}-2\nu\vh{\nabla_\Sigma\frac{1}{\nu},\nabla_\Sigma \dfrac{1}{\nu}}=\left(\Ric(N,N)+\td{A}^2\right)\dfrac{1}{\nu}\ge -\dfrac{ 2}{\nu}$$
since $\Ric_{\sol}\ge -2$ (see \cite{DHM09}).
Let us define $h=\eta\dfrac{1}{\nu}$ where $\eta$ is a positive function.
\begin{align*}
Lh=&L\left(\eta\dfrac{1}{\nu}\right)=\Delta_\Sigma\left(\eta\dfrac{1}{\nu}\right)-2\nu\vh{\nabla_\Sigma\dfrac{1}{\nu},\nabla_\Sigma\left(\eta\dfrac{1}{\nu}\right)}\\
=&\left(\eta\Delta_\Sigma\dfrac{1}{\nu}+2\vh{\nabla_\Sigma\eta,\nabla_\Sigma\dfrac{1}{\nu}}+\dfrac{1}{\nu}\Delta_\Sigma\eta\right)\\
&-2\nu\vh{\nabla_\Sigma\dfrac{1}{\nu},\eta\nabla_\Sigma\dfrac{1}{\nu}+\dfrac{1}{\nu}\nabla_\Sigma\eta}\\
=&\eta L\dfrac{1}{\nu}+\dfrac{1}{\nu}\Delta_\Sigma\eta\ge (\Delta_\Sigma\eta- 2\eta)\dfrac{1}{\nu}.
\end{align*}
We define on $\Sigma$ the function
$$\varphi(x)=\max\left\{ -\dfrac{u(x)}{2u(p)}+1-\varepsilon-\dfrac{d(x)^2}{R^2},0 \right\}$$
where $d=d(-,p)$. By definition, $$\varphi(p)=\dfrac{1}{2}-\varepsilon,\qquad 0\le\varphi\le 1-\varepsilon,\qquad {\rm supp}\varphi\subset\subset\Sigma.$$
We define $\eta=e^{K\varphi}-1.$ We calculate $\eta'(\varphi)=Ke^{K\varphi}$, $\eta''(\varphi)=K^2e^{K\varphi}.$ Let $q$ such that $$h(q)=\sup_\Omega h>0.$$ At the point $q$, we have

\begin{align*}
\Delta_\Sigma\eta- 2\eta&=\left(\eta'(\varphi)\Delta_\Sigma\varphi+\eta''(\varphi)\td{\nabla_\Sigma\varphi}^2\right)- 2\left(e^{K\varphi}-1\right)\\
&=e^{K\varphi}\left(K^2\td{\nabla_\Sigma\varphi}^2+K\Delta_\Sigma\varphi- 2\right)+ 2\\
&\ge e^{K\varphi}\left(K^2\td{\nabla_\Sigma\varphi}^2+K\Delta_\Sigma\varphi- 2\right).
\end{align*}
By the definition of $\varphi$,
\begin{align}\label{equ*}
\td{\nabla_\Sigma\varphi}^2&=\td{-\dfrac{\nabla_\Sigma u}{2u(p)}-\dfrac{\nabla_\Sigma d^2}{R^2}}^2=\td{\dfrac{\dhr_t^\top}{2u(p) y^2}+\dfrac{2d\dhr_d^\top}{R^2}}^2\notag\\
&=\dfrac{1}{4u(p)^2 y^2}\left(1-\dfrac{1}{W^2}\right)+\dfrac{4d^2}{R^4}\td{\dhr_d^\top}^2+\dfrac{2d}{u(p)R^2 y^2}\vh{\dhr_t^\top,\dhr_d^\top}\notag\\
&\ge \dfrac{1}{4u(p)^2 y^2}\left(1-\dfrac{1}{W^2}\right)+0-\dfrac{2d}{u(p)R^2 y^2}\nu\vh{\dhr_d,N}\notag\\
&=\dfrac{1}{4u(p)^2 y^2}\left(1-\dfrac{1}{W^2}-\dfrac{8 y u(p)}{R}\dfrac{d}{R}\vh{\dhr_d,N}\dfrac{1}{W}\right)\notag\\
&\ge \dfrac{1}{4u(p)^2 y^2}\left(1-\dfrac{1}{W^2}-\dfrac{8 y u(p)}{R}\dfrac{1}{W}\right)
\end{align}
since $d\le R$. Hence, if $\dfrac{1}{W}\le \min\left\{\dfrac{1}{2},\dfrac{R}{32 y u(p)}\right\}$, we have
$$\td{\nabla_\Sigma\varphi}^2\ge \dfrac{1}{8u(p)^2 y^2}.$$
Moreover
\begin{align}\label{equ**}
\Delta_\Sigma\varphi&=-\dfrac{\Delta_\Sigma u}{2u(p)}-\dfrac{\Delta_\Sigma d^2}{R^2}\notag\\
&=-\dfrac{1}{2u(p)}\left(\dfrac{2}{W y^2}\vh{\nabla y,N}\right)-\dfrac{\Delta_\Sigma d^2}{R^2}\notag\\
&=-\dfrac{1}{ y^2u(p)^2}\left(\dfrac{\vh{\nabla y,N}}{W}u(p)+\dfrac{ y^2\Delta_\Sigma d^2}{R^2}u(p)^2\right)\notag\\
&\ge -\dfrac{1}{ y^2u(p)^2}\left(C_1u(p)+\dfrac{C_2}{R^2}u(p)^2\right).
\end{align}
Combining \eqref{equ*} with \eqref{equ**} yields
\begin{multline*}
K^2\td{\nabla_\Sigma\varphi}^2+K\Delta_\Sigma\varphi- 2\\\ge \dfrac{1}{8u(p)^2 y^2}K^2-\dfrac{1}{ y^2u(p)^2}\left(C_1u(p)+\dfrac{C_2}{R^2}u(p)^2\right)K- 2\\
\ge \dfrac{1}{8u(p)^2 y^2}\left(K^2-8\left(C_1u(p)+\dfrac{C_2}{R^2}u(p)^2\right)K-8C_3u(p)^2 \right).
\end{multline*}

It follows that, if $K=\left(8C_1+\frac{C_3}{C_1}\right)u(p)+8\dfrac{C_2}{R^2}u(p)^2$, we obtain $K^2\td{\nabla_\Sigma\varphi}^2+K\Delta_\Sigma\varphi-1>0$, then, $Lh> 0.$ By Maximum principle applied to $L$, it implies that the maximum of $h$ can only be attained at a point $q$ where $\dfrac{1}{W(q)}\ge \min\left\{\dfrac{1}{2},\dfrac{R}{32 y u(p)}\right\}$.


\begin{align*}
\left(e^{K(\frac{1}{2}-\varepsilon)}-1\right)\dfrac{1}{\nu(p)}&=h(p)\le h(q)=\left(e^{K\varphi(q)}-1 \right)\dfrac{1}{\nu(q)}\\
&\le\dfrac{e^K-1}{\min\left\{\dfrac{ y(q)}{2},\dfrac{R}{32 u(p)}\right\}}.
\end{align*}
Letting $\varepsilon$ tending to $0$ we get
$$\nu(p)\ge \min\left\{\dfrac{ y(q)}{4},\dfrac{R}{64 u(p)}\right\}e^{-\frac{K}{2}}.$$
So
$$\td{\nabla u(p)}\le \max\left\{\dfrac{4}{ y(q)},\dfrac{64}{R}u(p)\right\}e^{\frac{1}{2}\left(\left(8C_1+\frac{C_3}{C_1}\right)u(p)+\frac{C_2}{R^2}u(p)^2\right)}
.$$
Then
$$\td{\nabla u(p)}\le e^{C(1+t)^2}$$
for $C=C(R)$ large enough.

\end{proof}

\subsection{Existence theorem}
In this subsection, we give a result concerning the existence of a solution of the Dirichlet problem for the minimal surface equation \eqref{MSE}.

By using interior gradient estimate (Theorem \ref{EGI}), elliptic estimate, and  Arzel\`a-Ascoli theorem, we obtain the compactness theorem as follows.

\begin{dl}[Compactness theorem]\label{TC}
Let $\{u_n\}$ be a sequence of minimal solutions on a domain $\Omega\subset\hh^2$. Suppose that $\{u_n\}$ is uniformly bounded on compact subsets of $\Omega$. Then there exists a subsequence of $\{u_n\}$ converging on compact subsets of $\Omega$ to a minimal solutions on $\Omega$.
\end{dl}

\begin{dl}\label{TEP}
Let $\Omega\subset\hh^2$ be a bounded domain with $\dhr\Omega\in C^2$. Suppose that $\Omega$ is Euclidean mean convex. Let $f\in C^0(\dhr\Omega)$ be a continuous function. Then there exists a unique minimal solution $u$  on $\Omega$ such that $u=f$ on $\dhr\Omega$.
\end{dl}

\begin{proof}
The uniqueness is deduced by Maximum principle, Theorem \ref{PMG}.

\textbf{Existence:} Let $\alpha,\beta$ be two real numbers such that $\alpha<f(x)<\beta$,  $x\in\dhr\Omega$.
Since $\Omega\subset\hh^2$ is a bounded Euclidean mean convex domain, $M^3:=\overline{\Omega}\times[\alpha,\beta]$ is a manifold of dimension $3$, compact, and mean convex. Define a Jordan curve $\sigma\subset \dhr M^3$
$$\sigma=\{(x,f(x)):x\in\dhr\Omega \}.$$
By Theorem of Meeks-Yau(see \cite[Theorem 1]{MY82}, \cite[Theorem 6.28]{CM11}), there exists a minimal surface $\overline\Sigma$
$$\dhr\overline{\Sigma}=\sigma,\qquad \Sigma:=\overline{\Sigma}\setminus \sigma\subset \Omega\times[\alpha,\beta].$$
Then, it is sufficient to show that $\Sigma$ is a graph. Suppose the contrary, that $\Sigma$ is not a graph. There exists a point $p\in\Sigma$ such that $\dhr_t|_p\in T_p\Sigma$. By Corollary \ref{cor2.2} there exists a unique Euclidean geodesic $\gamma$ such that two minimal surfaces $\Sigma$ and $\gamma\times\rr$ are tangents at $P$.

Since $\Sigma$ is not invariant by translation along $\dhr_t$, both two surfaces $\Sigma$, $\gamma\times\rr$ are not coincide. By Theorem of local description for the Intersections of minimal surfaces \cite[Theorem 7.3]{CM11}, in a neighborhood of $P$, the intersection of $\Sigma$ and $\gamma\times\rr$ composed of $2m\, (m\ge 2)$ arcs meeting at $P$.

If there exists a cycle $\alpha$ in $\Sigma\cap\gamma\times\rr$	, then $\alpha$ is the boundary of a minimal disk in $\Sigma$.
Thus we could touch this disk at an interior point with another minimal surface $\beta\times\rr$, where $\beta$ is an Euclidean geodesic curve of $\hh^2$, but this can not happen by the maximum principle. 
So each branch of these curves leaving $p$ must go to $\dhr\Sigma$ and, as $\gamma\cap\dhr\Omega$ has
exactly two points, at least two of the branches go to the same point of $\dhr\Sigma$. This yields a compact cycle $\alpha$ in $\overline{\Sigma}\cap(\gamma\times\rr)$ and, by the same previous
argument, we have a contradiction.
\end{proof}

A  function $u\in C^0(\Omega)$ will be called subsolution (resp. supersolution) in 
$\Omega$ if for every disk $D\subset\subset\Omega$ and every function $h$ minimal solution in $D$ satisfying $u\le h$ (resp. $u\ge h$) on $\dhr D$, we also have $u\le h$ (resp. $u\ge h$) in $D$. We will have the following properties of $C^0(\Omega)$ subsolution.

\begin{cy}\label{tc}
\begin{itemize}
\item[\rm (i)] A function $u\in C^2(\Omega)$ is a subsolution if and only if $\Mcal u\ge 0$.
\item[\rm (ii)] If $u$ is subsolution in a domain $\Omega$ and if $v$ is supersolution in a bounded domain $\Omega$ with $v\ge u$ on $\dhr\Omega$, then  $v\ge u$ on $\Omega$. To prove the latter assertion, suppose the contrary. Then at some point $p_0\in\Omega$ we have 
$$(u-v)(p_0)=\sup_{\Omega}(u-v)=M\ge 0$$
and we may assume there is a disk $D=\Dbb(p_0)$ such that $u-v\not\equiv M$ on $\dhr D$. Denote by
$\overline{u},\overline{v}$ the minimal solutions  respectively equal to $u, v$ on $\dhr D$ by Theorem \ref{TEP}, 
one sees that 
$$M\ge \sup_{\dhr D}(\overline{u}-\overline{v})\ge (\overline{u}-\overline{v})(p_0)\ge (u-v)(p_0)=M$$
and hence the equality holds throughout. By the maximum principle for minimal solution it follows that $\overline{u}-\overline{v}\equiv M$ in $D$ and hence $u-v=M$ on $\dhr D$, which contradicts the choice of $D$.
\item[\rm (iii)] Let $u$ be subsolution in $\Omega$ and $D$ be a disk strictly contained in $\Omega$. Denote by $\overline{u}$ the minimal solution in $D$  satisfying 
$\overline{u} = u$ on $\dhr D$. We define in $\Omega$ the minimal solution lifting of $u$ (in $D$) by 
\begin{equation}
U(p)=\begin{cases}
\overline{u}(p),&p\in D\\
u(p),& p\in\Omega\setminus D.
\end{cases}
\end{equation}
Then the function $U$ is also subsolution in $\Omega$. Indeed, consider an arbitrary disk $D'\subset\subset\Omega$ and let $h$ be a minimal solution  in $D'$ satisfying $h\ge U$ on $\dhr D'$. Since $u\le U$ in $D'$ we have $u\le h$ in $D'$ and hence $U\le h$ in $D'\setminus D$.  Since $U$ is minimal solution in $D$, we have by the maximum principle $U\le h$ in $D \cap D'$. Consequently $U\le h$ in $D'$ and $U$ is subsolution in $\Omega$. 

\item[\rm (iv)] 
Let $u_1,u_2,\ldots,u_N$ be subsolution in $\Omega$. 
Then the function $u(p) =\max \{u_1(p),\ldots,u_N(p)\}$ is also subsolution in $\Omega$. 
This is a trivial consequence of the  definition of subsolution. 
Corresponding results for supersolution functions 
are obtained by replacing $u$ by $- u$ in properties (i), (ii), (iii) and (iv). 
\end{itemize}
\end{cy}

Now let $\Omega$ be bounded domain and $f $ be a bounded function on $\dhr\Omega$. A function $u\in C^0(\overline{\Omega})$ will be called a subfunction (resp. superfunction) relative to $f$ if $u$ is a subsolution (resp. supersolution) in $\Omega$ and $u\le f$ (resp. $u\ge f$) on $\dhr\Omega$. By  \ref{tc}(ii), every subfunction is less than or equal to every superfunction. In particular, constant functions $\le\inf_\Omega f $ (resp. $\ge\sup_\Omega f )$ 
are subfunctions (resp. superfunctions). Denote by $S_f $ the set of subfunctions relative to $f $. The basic result of the Perron method is contained in the following theorem. 

\begin{md}\label{Perron1}
The function $u(p)= \sup_{v\in S_f } v(p)$ is a  minimal solution in $\Omega$. 
\end{md}

\begin{proof}
By the maximum principle any function $v\in S_f $ satisfies $v\le\sup_{\dhr\Omega}f $, so that $u$ is well defined. 
Let $q$ be an arbitrary fixed point of $\Omega$. By the definition of $u$, there exists a sequence $\{v_n\}\subset S_f $ such that $v_n(q)\to u(q)$. By replacing $v_n$ with $\max\{v_n,\inf f \}$, 
we may assume that the sequence $\{v_n\}$ is bounded. Now choose $R$ so that the disk $D=\Dbb_R(q)\subset\subset\Omega$ and define $V_n$ to be the minimal solution lifting of $v_n$ in $D$ according 
to (iii). 
Then $V_n\in S_f , V_n(q)\to u(q)$ and by Theorem \ref{TC} the sequence $\{V_n\}$ contains a subsequence $\{V_{n_k}\}$ converging uniformly in any disk $\Dbb_\rho(q)$ with $\rho<R$ to a function $v$ that is minimal solution in $D$. Clearly $v\le u$ in $D$ and $v(q)=u(q)$. We claim 
now that in fact $v=u$ in $D$. For suppose $v(\overline{q})<u(\overline{q})$ at some $\overline{q}\in D$. Then there exists a function $\overline{u}\in S_f $ such that $v(\overline{q})<\overline{u}(\overline{q})$. Defining $w_k=\max\{\overline{u},V_{n_k}\}$ and also the minimal solution liftings $W_k$ as in (iii), we obtain as before a subsequence of the sequence $\{W_k\}$ converging to a minimal solution function $w$ satisfying $v\le w\le u$ in $D$ and 
$v(q)=w(q)=u(q)$. But then by the maximum principle we must have $v=w$ in $D$. This contradicts the definition of $\overline{u}$ and hence $u$ is minimal solution in $\Omega$. 
\end{proof}

We will show the solution that we obtained (called the Perron solution) will be the solution of the Dirichlet  problem as follows.

\begin{dl}\label{TEG}
Let $\Omega$ be a bounded addmissible domain with $\{C_i\}$ the open arcs of $\dhr\Omega$. Let $f_i\in C^0(C_i)$ be bounded functions. Assume $C_i$ are Euclidean mean convex to $\Omega$ then there exists a unique minimal solution $u$  on $\Omega$ such that $u=f_i$ on $C_i$ for all $i$.
\end{dl}
\begin{proof}
Let a function $f$ defined on $\dhr\Omega$ such that $f(p)=f_i(p)$ if $p\in C_i$.
Denote by $u$ the Perron solution relative to $\Mcal$ and $f$.
 Fix $\xi\in C_i$, for some $i$. We must prove that
\begin{equation}\label{2}
\lim_{p\in\Omega,p\to\xi}u(p)=f(\xi).
\end{equation}

We construct the local barrier at $\xi$ as follows.
For $r>0$ small enough, consider the domain $\Omega\cap \Dbb_r(\xi)$.  We approximate $\Omega\cap \Dbb_r(\xi)$ by $C^2$ (Euclidean mean convex) domain $\Omega_\xi\subset\Omega\cap \Dbb_r(\xi)$  by rounding each
corner point of $\Omega\cap B_r(\xi)$.
By Theorem \ref{TEP}, there exist minimal solutions  $w_{\pm}\in C^2(\Omega_\xi)\cap C^0(\overline{\Omega_\xi})$ on $\Omega_\xi$  such that  $w_\pm(\xi)=f(\xi)$ and 
$$\begin{cases}
w_-\le f\le w_+ & \text{on } \dhr\Omega_\xi\cap\dhr\Omega,\\
w_-\le\inf f\le \sup f\le w_+ &\text{on } \dhr\Omega_\xi\cap\Omega.
\end{cases}$$
From the definition of $u$ and the fact that every subfunction is dominated by every superfunction, we have
$$w_-\le u\le w_+,\qquad\text{on }\Omega_\xi,$$
 we obtain (\ref{2}).
\end{proof}

\section{A local Scherk surface in $\sol$ and Flux formula}\label{sec4}

\subsection{A local Scherk surface in $\sol$}

\begin{md}\label{scherk2}
Let $\Omega\subset\hh^2$ be an Euclidean mean convex quadrilateral domain whose boundary $\dhr\Omega$ is composed of two Euclidean geodesic arcs $A_1, A_2$ and two Euclidean geodesic arcs $C_1,C_2$. 
Suppose that 
\begin{equation}\label{cond0}
\ell_\euc(A_1)+\ell_\euc(A_2)< \ell_\euc(C_1)+\ell_\euc(C_2).
\end{equation}
Let $f_i$ be a positive continuous function on $C_i$, $i=1,2.$ Then there exists a minimal solution $u$ in $\Omega$ taking $+\infty$ on $A_i$ and $f_i$ on $C_i$.
\end{md}
This construction was motivated by \cite[Theorem 2]{NR02}.
\begin{proof}
This proof is divided into two cases.
\begin{trh} Case $f_1=f_2=0$.
\end{trh}
\begin{proof}
Let $n$ be a fixed positive number.
By Theorem \ref{TEG}, there exists a minimal  solution $u_n$ in $\Omega$ taking $n$ on $A_i$ and $0$ on $C_i$.
By General maximum principle (Theorem \ref{PMG}), $0\le u_n\le u_{n+1}$.\\
We will prove that the sequence $\{u_n\}$ is uniformly bounded on compact subsets $K$ of  $\Omega\cup C_1\cup C_2.$ \\
We first  construct  minimal annulus.

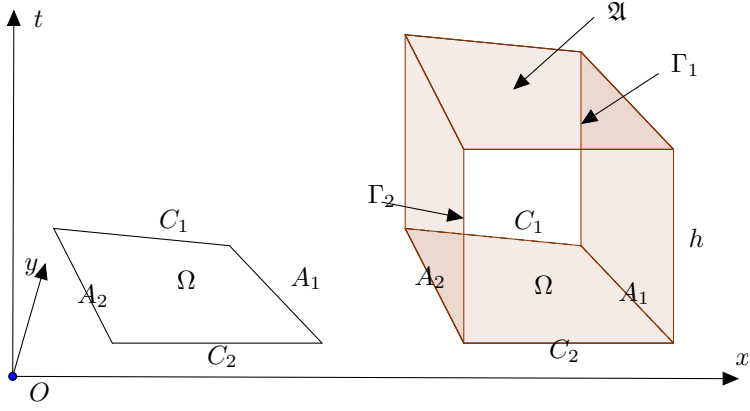
\begin{figure}[!h]
\centering

\definecolor{qqqqff}{rgb}{0,0,1}
\definecolor{zzttqq}{rgb}{0.6,0.2,0}
\begin{tikzpicture}[line cap=round,line join=round,>=triangle 45,x=0.7788161993769468cm,y=0.7627118644067798cm]
\clip(-4,-1.14) rectangle (8.84,5.94);
\fill[color=zzttqq,fill=zzttqq,fill opacity=0.1] (3,5.38) -- (4,3.38) -- (4,0) -- (3,2) -- cycle;
\fill[color=zzttqq,fill=zzttqq,fill opacity=0.1] (3,5.38) -- (6,5.08) -- (7.58,3.38) -- (4,3.38) -- cycle;
\fill[color=zzttqq,fill=zzttqq,fill opacity=0.1] (6,5.08) -- (6,1.7) -- (6,1.7) -- (7.58,0) -- (7.58,3.38) -- cycle;
\fill[color=zzttqq,fill=zzttqq,fill opacity=0.1] (3,2) -- (6,1.7) -- (7.58,0) -- (4,0) -- cycle;
\draw (-3,2)-- (-2,0);
\draw (-2,0)-- (1.58,0);
\draw (-3,2)-- (0,1.7);
\draw (0,1.7)-- (1.58,0);
\draw (3,2)-- (4,0);
\draw [dash pattern=on 3pt off 3pt] (3,2)-- (6,1.7);
\draw (4,0)-- (7.58,0);
\draw [dash pattern=on 3pt off 3pt] (6,1.7)-- (7.58,0);
\draw (3,5.38)-- (4,3.38);
\draw (3,5.38)-- (6,5.08);
\draw (6,5.08)-- (7.58,3.38);
\draw (4,3.38)-- (7.58,3.38);
\draw [color=zzttqq] (3,5.38)-- (4,3.38);
\draw [color=zzttqq] (4,3.38)-- (4,0);
\draw [color=zzttqq] (4,0)-- (3,2);
\draw [color=zzttqq] (3,2)-- (3,5.38);
\draw [color=zzttqq] (3,5.38)-- (6,5.08);
\draw [color=zzttqq] (6,5.08)-- (7.58,3.38);
\draw [color=zzttqq] (7.58,3.38)-- (4,3.38);
\draw [color=zzttqq] (4,3.38)-- (3,5.38);
\draw [color=zzttqq] (6,5.08)-- (6,1.7);
\draw [color=zzttqq] (6,1.7)-- (6,1.7);
\draw [color=zzttqq] (6,1.7)-- (7.58,0);
\draw [color=zzttqq] (7.58,0)-- (7.58,3.38);
\draw [color=zzttqq] (7.58,3.38)-- (6,5.08);
\draw [color=zzttqq] (3,2)-- (6,1.7);
\draw [color=zzttqq] (6,1.7)-- (7.58,0);
\draw [color=zzttqq] (7.58,0)-- (4,0);
\draw [color=zzttqq] (4,0)-- (3,2);
\draw (0.88,1.48) node[anchor=north west] {$A_1$};
\draw (-2.76,1.18) node[anchor=north west] {$A_2$};
\draw (-1.36,2.48) node[anchor=north west] {$C_1$};
\draw (-0.54,0.12) node[anchor=north west] {$C_2$};
\draw (6.48,1.22) node[anchor=north west] {$A_1$};
\draw (4.7,2.46) node[anchor=north west] {$C_1$};
\draw (5.3,0.22) node[anchor=north west] {$C_2$};
\draw (3,1.5) node[anchor=north west] {$A_2$};
\draw (7.68,2.16) node[anchor=north west] {$h$};
\draw (7.38,5.2) node[anchor=north west] {$\Gamma_1$};
\draw (2.2,2.92) node[anchor=north west] {$\Gamma_2$};
\draw (6.3,6.12) node[anchor=north west] {$\mathfrak{A}$};
\draw [->] (2.6,2.46) -- (4,2.18);
\draw [->] (7.32,4.7) -- (6,3.82);
\draw [->] (6.22,5.68) -- (4.84,4.42);
\draw (-1.06,1.42) node[anchor=north west] {$\Omega$};
\draw (5.04,1.34) node[anchor=north west] {$\Omega$};
\draw [->] (-3.7,-0.58) -- (-3.68,5.82);
\draw [->] (-3.7,-0.58) -- (-3.14,1.4);
\draw (-3.58,-0.54) node[anchor=north west] {$O$};
\draw (8.48,-0.02) node[anchor=north west] {$x$};
\draw (-3.5,5.96) node[anchor=north west] {$t$};
\draw (-3.66,1.64) node[anchor=north west] {$y$};
\draw [->] (-3.7,-0.58) -- (8.7,-0.62);
\begin{scriptsize}
\draw [fill=qqqqff] (-3.7,-0.58) circle (1.5pt);
\end{scriptsize}
\end{tikzpicture}
\caption{Annulus  $\mathfrak{A}$ }\label{F3}
\end{figure}

Fix $h\in\rr,h>0$ and let $\Gamma_i$ be the curves that are the boundary of $C_i\times[0,h]$. Let $\Sigma_i^h$ be a minimal disk with boundary $\Gamma_i$. Then
$$\Area(\Sigma_i^h) = \Area(C_i\times[0,h])= h\cdot\ell_{\euc}(C_i).$$
Consider the annulus $\mathfrak{A}$ with boundary $\Gamma_1\cup\Gamma_2$ (see Figure \ref{F3}):
$$\mathfrak{A}=\Omega\cup T_h(\Omega)\cup\bigcup_{i=1}^2(A_i\times[0,h])$$ where $T_h$ is defined by (\ref{equ2.5}). Then
$$\Area(\mathfrak{A} )=2\Area(\Omega)+h(\ell_{\euc}(A_1)+\ell_{\euc}(A_2)).$$
Therefore
\begin{align*}
\Area(\mathfrak{A} )-\left(\Area(\Sigma_1^h)+\Area(\Sigma_2^h)\right)\le & 2\Area(\Omega)+h(\ell_{\euc}(A_1)+\ell_{\euc}(A_2)\\
&-\ell_{\euc}(C_1)-\ell_{\euc}(C_2)).
\end{align*}
By the hypothesis (\ref{cond0}), $\Area(\mathfrak{A} )-\left(\Area(\Sigma_1^h)+\Area(\Sigma_2^h)\right)<0$
if $h\ge h_0$ where $h_0$ is sufficiently large.
Hence, $\Area(\mathfrak{A})$ is smaller than the sum of the areas of the disks $C_i\times[0,h]$, and by the Douglas criteria \cite{Jos85}, there exists a least area minimal annulus $\mathfrak{A}(h)$ with boundary $\Gamma_1\cup \Gamma_2$ for all $h\ge h_0$.

\begin{kd}
The annulus $\mathfrak{A}(h)$ is an upper barrier for the sequence $\{\Gr(u_n)\}$  for all $n>0$ and $h\ge h_0$. Moreover, the vertical projections of the annulus $\mathfrak{A}(h)$ is an exhaustion for $\Omega\cup C_1\cup C_2$.
\end{kd}
\begin{proof}
For the proof we refer the reader to \cite[page 271, 272]{NR02} or \cite[page 126, 127]{Pin09}.
\renewcommand{\qedsymbol}{$\Diamond$}
\end{proof}
In this assertion we conclude that the sequence $\{u_n\}$ is uniformly bounded on compact subsets of $\Omega\cup C_1\cup C_2.$ 
By the compactness (Theorem \ref{TC}), the sequence $\{u_n\}$ converges on compact subsets of $\Omega$ to a minimal solution $u$ on $\Omega$ which assumes the above prescribed boundary values on $\dhr\Omega$.
\renewcommand{\qedsymbol}{$\triangle$}
\end{proof}

\begin{trh} General case.
\end{trh}
\begin{proof}
For every $n>0$, by applying Theorem \ref{TEG}, there exists a minimal solution $u_n$ on $\Omega$ with boundary values
$$u_n|_{A_i}=n,\qquad u_n|_{C_i}=\min\{n,f_i\}.$$
By Maximum principle (Theorem \ref{PMG}), $u_n\le u_{n+1}$.

\begin{kd}
The sequence $u_n$ is uniformly bounded on every compact subset $K$ of $\Omega\cup C_1\cup C_2.$
\end{kd}
\begin{proof}
Denote by $K$ a compact subset of $\Omega\cup C_1\cup C_2$. Then $\varepsilon:=\dist(K,A_1\cup A_2)>0.$ We define a subdomain $\Omega'$ of $\Omega$ by the formula
$$\Omega'=\left\{p\in\Omega:\dist(p,A_1\cup A_2)>\frac{\varepsilon}{2} \right\}.$$
Let us denote $C'_i=C_i\cap\dhr\Omega'$ and $A'_1\cup A'_2=A':=\Omega\cap\dhr\Omega'$. (See Figure \ref{F4}).
It follows from the definition that $K$ is a compact subset of $\Omega'\cup C'_1\cup C'_2$.
There is, by the previous case, a minimal solution $w$ on $\Omega'$ which obtain the values $+\infty$ on $A'_i$ and $0$ on $C'_i$.


\begin{figure}[!h]
\centering
\definecolor{yqyqyq}{rgb}{0.5,0.5,0.5}
\definecolor{dcrutc}{rgb}{0.86,0.08,0.24}
\begin{tikzpicture}[line cap=round,line join=round,>=triangle 45,x=1.0cm,y=1.0cm]
\clip(-0.82,1.12) rectangle (4.78,5.52);
\fill[color=dcrutc,fill=dcrutc,fill opacity=0.1] (1.26,3.46) -- (1.74,3.92) -- (2.42,3.58) -- (2.4,2.94) -- (1.62,2.86) -- cycle;
\fill[color=yqyqyq,fill=yqyqyq,fill opacity=0.1] (0.61,4.85) -- (0.64,2.08) -- (3.4,2.43) -- (3.39,4.15) -- cycle;
\draw (0,5)-- (0,2);
\draw (0,2)-- (4,2.5);
\draw (4,2.5)-- (4,4);
\draw (4,4)-- (0,5);
\draw [color=dcrutc] (1.26,3.46)-- (1.74,3.92);
\draw [color=dcrutc] (1.74,3.92)-- (2.42,3.58);
\draw [color=dcrutc] (2.42,3.58)-- (2.4,2.94);
\draw [color=dcrutc] (2.4,2.94)-- (1.62,2.86);
\draw [color=dcrutc] (1.62,2.86)-- (1.26,3.46);
\draw [color=yqyqyq] (0.61,4.85)-- (0.64,2.08);
\draw [color=yqyqyq] (0.64,2.08)-- (3.4,2.43);
\draw [color=yqyqyq] (3.4,2.43)-- (3.39,4.15);
\draw [color=yqyqyq] (3.39,4.15)-- (0.61,4.85);
\draw (2.08,5.62) node[anchor=north west] {$C_1$};
\draw (1.94,1.82) node[anchor=north west] {$C_2$};
\draw (1.72,3.78) node[anchor=north west] {$K$};
\draw (1.02,4.44) node[anchor=north west] {$\Omega'$};
\draw (4.02,3.52) node[anchor=north west] {$A_1$};
\draw (-0.6,3.84) node[anchor=north west] {$A_2$};
\draw (3.08,3.82) node[anchor=north west] {$A'_1$};
\draw (0.32,4) node[anchor=north west] {$A'_2$};
\draw (1.96,4.9) node[anchor=north west] {$C'_1$};
\draw (1.94,2.56) node[anchor=north west] {$C'_2$};
\end{tikzpicture}
\caption{The domain $\Omega'$}\label{F4}
\label{fig4}
\end{figure}
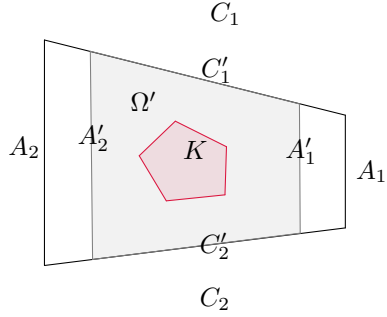

By the general maximum principle (Theorem \ref{PMG}), we have $0\le u_n\le w+\sum_{i=1}^{2}\max_{C'_i}f_i$ on $\Omega'\cup C'_1\cup C'_2$. Since $K$ is a compact of $\Omega'\cup C'_1\cup C'_2$, $\{u_n\}$ is uniformly bounded on $K$.
\renewcommand{\qedsymbol}{$\Diamond$}
\end{proof}
It follows from the previous affirmation and the compactness theorem (Theorem \ref{TC}) that, the sequence $\{u_n\}$ converges on each compact subset of $\Omega\cup C_1\cup C_2$ to a solution $u$ on $\Omega$. Moreover, we have $u|_{C_i}=\lim_n u_n|_{C_i}=f_i$ and $u|_{A_i}=\lim_n u_n|_{A_i}=+\infty.$  This completes the proof.
\renewcommand{\qedsymbol}{$\triangle$}
\end{proof}

\end{proof}

\begin{md}\label{barrier}
Let $\Omega\subset\hh^2$ be a bounded domain whose boundary $\dhr\Omega$ is composed of an Euclidean geodesic arc  $A$ and an Euclidean convex arc $C$ with their endpoints. Then, there exists a minimal solution $u$ in $\Omega$ taking $+\infty$ on $A$ and arbitrarily positive continuous function $f$ on $C$.
\end{md}
\begin{proof}
For every $n>0$, by applying Theorem \ref{TEG}, there is a minimal solution $u_n$ on $\Omega$ with boundary values
$$u_n|_{A}=n,\qquad u_n|_{C}=\min\{n,f\}.$$
By Maximum principle, Theorem \ref{PMG}, $0\le u_n\le u_{n+1}$ for every $n$.

\begin{kd}
The sequence $\{u_n\}$ is uniformly bounded on every compact subset $K$ of $\Omega\cup C.$
\end{kd}
\begin{proof}

Denote by $K$ a compact subset of $\Omega\cup C$. Then $\varepsilon:=\dist(K,A)>0.$ We define a subdomain $\Omega'$ of $\Omega$ by the formula
$$\Omega'=\left\{p\in\Omega:\dist(p,A)>\frac{\varepsilon}{2} \right\}.$$
Let us denote $C'=C\cap\dhr\Omega'$ and $A'=\Omega\cap\dhr\Omega'$. (See Figure \ref{F5}).
It follows from the definition that $K$ is a compact subset of $\Omega'\cup C'$.
By Theorem \ref{scherk2}, there is a quadrilateral and a minimal solution $w$ defined on this quadrilateral that the values of $w$ on its boundary are $+\infty$ and $0$ (see Figure \ref{fig5}). 


\begin{figure}[!h]
\centering

\definecolor{wqwqwq}{rgb}{0.38,0.38,0.38}
\definecolor{ffqqff}{rgb}{1,0,1}
\begin{tikzpicture}[line cap=round,line join=round,>=triangle 45,x=1.0cm,y=1.0cm]
\clip(-2.34,-0.26) rectangle (7.32,5.4);
\fill[color=ffqqff,fill=ffqqff,fill opacity=0.1] (2,4) -- (3.3,3.74) -- (3,3) -- (2,2.74) -- (1.18,3.38) -- cycle;
\fill[color=wqwqwq,fill=wqwqwq,fill opacity=0.1] (-2,4.52) -- (7,4.48) -- (4.06,0.16) -- (2.32,0.16) -- cycle;
\draw [shift={(2.36,3.74)}] plot[domain=-3.89:0.73,variable=\t]({1*1.85*cos(\t r)+0*1.85*sin(\t r)},{0*1.85*cos(\t r)+1*1.85*sin(\t r)});
\draw (1,5)-- (3.74,4.98);
\draw [color=ffqqff] (2,4)-- (3.3,3.74);
\draw [color=ffqqff] (3.3,3.74)-- (3,3);
\draw [color=ffqqff] (3,3)-- (2,2.74);
\draw [color=ffqqff] (2,2.74)-- (1.18,3.38);
\draw [color=ffqqff] (1.18,3.38)-- (2,4);
\draw (2.1,3.8) node[anchor=north west] {$K$};
\draw (2.16,5.48) node[anchor=north west] {$A$};
\draw (2.34,2.04) node[anchor=north west] {$C$};
\draw (3.18,4.74) node[anchor=north west] {$A'$};
\draw (0.86,4.32) node[anchor=north west] {$\Omega'$};
\draw (5.3,2.28) node[anchor=north west] {$0$};
\draw (-0.2,2.54) node[anchor=north west] {$0$};
\draw (4.94,4.9) node[anchor=north west] {$\infty$};
\draw [color=wqwqwq] (-2,4.52)-- (7,4.48);
\draw [color=wqwqwq] (7,4.48)-- (4.06,0.16);
\draw [color=wqwqwq] (4.06,0.16)-- (2.32,0.16);
\draw [color=wqwqwq] (2.32,0.16)-- (-2,4.52);
\draw (3.28,0.16) node[anchor=north west] {$\infty$};
\end{tikzpicture}
\caption{The subdomain $\Omega'$ of $\Omega$}\label{F5}
\label{fig5}
\end{figure}
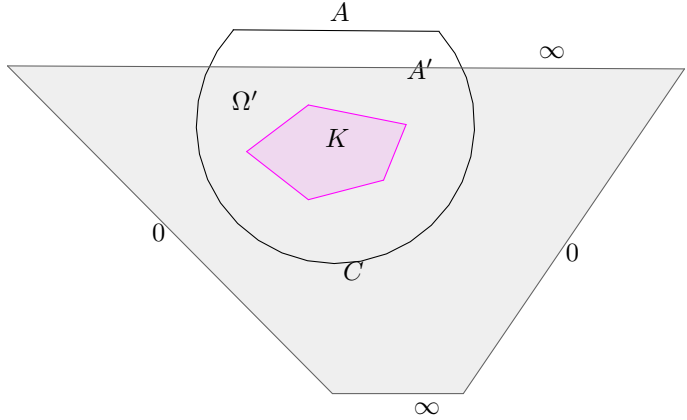

By General maximum principle (Theorem \ref{PMG}), we have $0\le u_n\le w+\max_{C'}f$ on $\Omega'\cup C'$. Since $K$ is a compact subset of $\Omega'\cup C'$, $\{u_n\}$ is uniformly bounded on $K$.
\renewcommand{\qedsymbol}{$\Diamond$}
\end{proof}
It follows from the previous affirmation, the compactness theorem \ref{TC} and the monotonicity of the sequence $\{u_n\}$, that the sequence $\{u_n\}$ converges on every compact subset of $\Omega\cup C$ to a minimal solution $u$ on $\Omega$. Moreover, we have $u|_C=\lim_n u_n|_C=f$ and $u|_A=\lim_n u_n|_A=+\infty.$ This completes the proof.
\end{proof}

\begin{bd}\label{lem-droite0}
Let $\Omega\subset\hh^2$ be a bounded domain whose boundary $\dhr\Omega$ is composed of an Euclidean geodesic arc  $A$ and an open Euclidean convex arc $C$ with their endpoints. Let $K$ be a compact subset of $\Omega\cup C$.  There exists a real number $M=M(K)$ such that if $u$ is a minimal solution on $\Omega$ that satisfies $u\ge c$ (resp. $u\le c$) on $C$ ($c$ is some real number), then $u\ge c-M$ (resp. $u\le c+M$) on $K$.
\end{bd}
\begin{proof} Since $K$ is a compact set of $\Omega\cup C$, $\varepsilon:=\dist(A,K)>0$. Define $\Omega'=\{x\in\Omega:\dist(x,A)>\frac{\varepsilon}{2}\}$. We have $\dhr\Omega'=A'\cup C'$ where $A':=\dhr\Omega'\cap\Omega$ an Euclidean geodesic arc and $C':=\dhr\Omega'\cap C$ a sub-arc of $C$. It follows from definition that $K$ is a compact set of $\Omega'\cup C'$.

By Proposition \ref{barrier}, there exists a minimal solution $w$ on $\Omega'$ such that $w|_{A'}=+\infty$ and $w|_{C'}=0$.  Define $M=\sup_Kw<\infty$, by the general maximum theorem, Theorem \ref{PMG}, we have $u\ge c-w$ (resp. $u\le c+w$) on $\Omega'$. So, $u\ge c-M$ (resp. $u\le c+M$) on $K$. This completes the proof.
\end{proof}

\begin{hq}[Straight line lemma]\label{lem-droite}
Let $\Omega\subset\hh^2$  be a domain, let $C\subset\dhr\Omega$ be an Euclidean mean convex arc (convex towards $\Omega$) and $u$ be a minimal solution in $\Omega$. If $u$ diverges to $+\infty$ or $-\infty$ as one approach $C$ within $\Omega$, then $C$ is an Euclidean geodesic arc.
\end{hq}
\begin{proof} Assume the contrary, that there exists a minimal solution $u$ sur $\Omega$ that takes the value $+\infty$ on $C$ where $C$ is not an Euclidean geodesic arc.

Take an strictly Euclidean mean convex subarc $C'$ of $C$. Let $\Gamma(C')$ be an Euclidean geodesic arc of $\hh^2$ joining the endpoints of $C'$. Denote by $\Omega'$ the domain delimited by $C'\cup\Gamma(C')$. We can choose $C'$ such that $\Omega'\subset\Omega$. (See Figure \ref{F6}).

\begin{figure}[!h]
\centering
\begin{tikzpicture}[line cap=round,line join=round,>=triangle 45,x=1.0cm,y=1.0cm]
\clip(-1.86,1.92) rectangle (5.2,6.36);
\draw [shift={(1.74,3.1)}] plot[domain=0.21:3.46,variable=\t]({1*2.7*cos(\t r)+0*2.7*sin(\t r)},{0*2.7*cos(\t r)+1*2.7*sin(\t r)});
\draw (-0.48,4.62)-- (2.46,5.69);
\draw (0.98,6.3) node[anchor=north west] {$C'$};
\draw (0.84,5.22) node[anchor=north west] {$\Gamma(C')$};
\draw (-1.24,5.78) node[anchor=north west] {$\Omega'$};
\draw [->] (-0.84,5.28) -- (0.74,5.24);
\begin{scriptsize}
\draw [fill=black] (1.16,5.48) circle (1.0pt);
\draw[color=black] (1.3,5.52) node {$q$};
\end{scriptsize}
\end{tikzpicture}
\caption{The domain $\Omega'$ and the arcs $C', \Gamma(C')$.}\label{F6}
\end{figure}
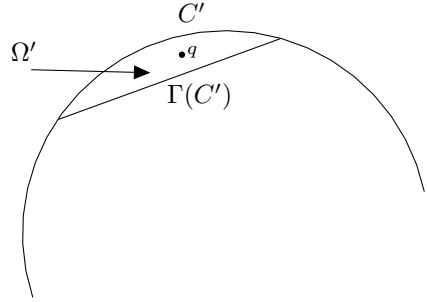
Let $q$ be a point in $\Omega'$. It follows from the lemma \ref{lem-droite0} that there exists a real number $M=M(q)$ such that $u\ge c-M$ for all real number $c$, a contradiction.
\end{proof}

\begin{dl}[Boundary values lemma \text{\cite[page 1882]{CR10}}]\label{LVB}

Let $\Omega\subset\hh^2$ be a domain and let $C$ be an Euclidean mean convex arc in $\dhr\Omega$. Suppose $\{u_n\}$ is a sequence of solutions in $\Omega$ that converges uniformly on every compact subset of $\Omega$ to a minimal solution $u$. Suppose each $u_n\in C^0(\Omega\cup C)$ and $u_n|_C$ converges uniformly on every compact subset of $C$ to a function $f$ on $C$ where $f$ is continuous or $f\equiv+\infty$ or $-\infty$. Then $u$ is continuous on $\Omega\cup C$ and $u|_C=f$.
\end{dl}
\begin{proof} It is sufficient to show that, for $p\in C$ and $M\in\rr$ such that $f(p)>M$, there exists a neighborhood $U$ of $p$ in $\Omega\cup C$ that satisfies $u>M$ on $U$.

Let $M'$ such that $M<M'<f(p)$. Since $f$ is continuous (or $f\equiv\infty$) and $u_n|_C$ converges uniformly on every compact subset of $C$ to $f$, there is a neighborhood $C'$ of $p$ in $C$ and a positive natural number $N_0$ such that $u_n(x)>M'$ for every $x\in C'$ and for every $n\ge N_0$. Consider two cases as follows.

(i) \textit{If $C$ is strictly Euclidean mean convex in a neighborhood of $p$ in $C$.} Without loss of generality, we suppose that $C'$ is strictly Euclidean mean convex.

Denote by $\Gamma(C')$ an open Euclidean geodesic arc of $\hh^2$ joining the end-points of $C'$. We can choose $C'$ such that $\Gamma(C')\subset\Omega$. Denote by $\Omega'$ the domain delimited by $C'\cup\Gamma(C')$.
(See Figure \ref{F7}).

\begin{figure}[!h]
\centering
\begin{tikzpicture}[line cap=round,line join=round,>=triangle 45,x=1.0cm,y=1.0cm]
\clip(-0.86,0.62) rectangle (5.54,5.14);
\draw [shift={(2.46,1.74)}] plot[domain=0.21:3.46,variable=\t]({1*2.7*cos(\t r)+0*2.7*sin(\t r)},{0*2.7*cos(\t r)+1*2.7*sin(\t r)});
\draw (0.24,3.26)-- (3.18,4.33);
\draw (1.7,5) node[anchor=north west] {$C'$};
\draw (1.56,3.86) node[anchor=north west] {$\Gamma(C')$};
\draw (-0.52,4.42) node[anchor=north west] {$\Omega'$};
\draw [->] (-0.12,3.92) -- (1.46,3.88);
\begin{scriptsize}
\draw [fill=black] (1.38,4.22) circle (1.0pt);
\draw[color=black] (1.46,4.42) node {$p$};
\end{scriptsize}
\end{tikzpicture}
\caption{The domain $\Omega'$ when $C'$ is Euclidean mean convex.}\label{F7}
\end{figure}
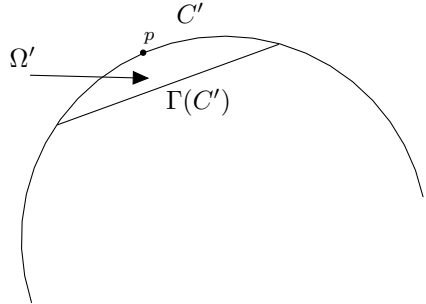


By Proposition \ref{barrier}, there exists a minimal solution $w$ on $\Omega'$ such that $w|_{C'}=M'$ and $w|_{\Gamma(C')}=-\infty$. It follows from the general maximum principle (Theorem \ref{PMG}), that $u_n\ge w$ on $\Omega'$ for every $n\ge N_0$. Hence we have $u\ge w$ sur $\Omega'$. Since $w$ is continuous, there is a neighborhood $U$ of $p$ in $\overline{\Omega'}$ such that $w>M$ on $U$. Therefore $u>M$ on $U$.

(ii) \textit{If the arc $C$ contains an Euclidean geodesic segment in a neighborhood of $p$.} Without loss of generality, we suppose that $C'$ is an Euclidean geodesic arc.

Consider a quadrilateral $\Pcal\subset\Omega$ such that $\dhr\Pcal$ is composed $4$ Euclidean geodesics $B_1,C_1,B_2,C_2$ where $C_1=C'$ and $ \ell_\euc(B_1)+\ell_\euc(B_2)<\ell_\euc(C_1)+\ell_\euc(C_2)$. (See Figure \ref{F8}).

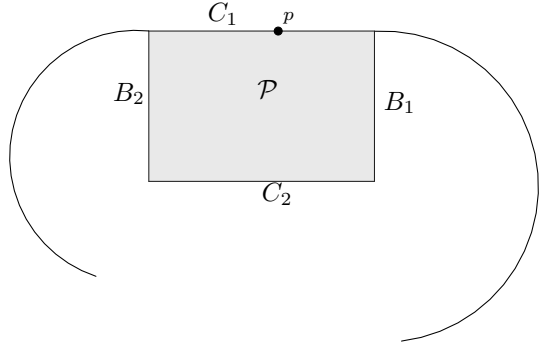
\begin{figure}[!h]
\centering
\definecolor{sqsqsq}{rgb}{0.13,0.13,0.13}
\begin{tikzpicture}[line cap=round,line join=round,>=triangle 45,x=1.0cm,y=1.0cm]
\clip(-1.1,0.56) rectangle (6.46,5.64);
\fill[color=sqsqsq,fill=sqsqsq,fill opacity=0.1] (1,5) -- (4,5) -- (4,3) -- (1,3) -- cycle;
\draw [color=sqsqsq] (1,5)-- (4,5);
\draw [color=sqsqsq] (4,5)-- (4,3);
\draw [color=sqsqsq] (4,3)-- (1,3);
\draw [color=sqsqsq] (1,3)-- (1,5);
\draw [shift={(4.11,2.93)}] plot[domain=-1.45:1.62,variable=\t]({1*2.07*cos(\t r)+0*2.07*sin(\t r)},{0*2.07*cos(\t r)+1*2.07*sin(\t r)});
\draw [shift={(0.83,3.33)}] plot[domain=1.47:4.39,variable=\t]({1*1.68*cos(\t r)+0*1.68*sin(\t r)},{0*1.68*cos(\t r)+1*1.68*sin(\t r)});
\draw (1.66,5.5) node[anchor=north west] {$C_1$};
\draw (2.38,3.1) node[anchor=north west] {$C_2$};
\draw (4,4.32) node[anchor=north west] {$B_1$};
\draw (0.4,4.44) node[anchor=north west] {$B_2$};
\draw (2.32,4.48) node[anchor=north west] {$\mathcal P$};
\begin{scriptsize}
\draw [fill=black] (2.72,5) circle (1.5pt);
\draw[color=black] (2.86,5.18) node {$p$};
\end{scriptsize}
\end{tikzpicture}
\caption {The domain $\Omega'$ when $C'$ is Euclidean geodesic.}\label{F8}
\end{figure}

Since $u_n$ converges uniformly on each compact subset of $\Omega$ to $u$, $M'':=\inf_{x\in C_2,n\ge 1}u_n(x)>-\infty$. By Proposition \ref{scherk2}, there is a minimal solution $w$ on $\Pcal$ such that $w|_{C_1}=M',w|_{C_2}=M''$ and $w=-\infty$ on $B_1\cup B_2$. It follows from the general maximum principle, Theorem \ref{PMG}, that $u_n\ge w$ on $\Omega'$ for every $n\ge N_0$. Hence we have $u\ge w$ on $\Omega'$. Since $w$ is continuous, there exists a neighborhood $U$ of $p$ in $\overline{\Omega'}$ such that $w>M$ on $U$. Then $u>M$ on $U$.

\end{proof}

\subsection{Flux formula}

Let $u$ be a minimal graph on a domain $\Omega\subset\hh^2$ . It follows from definition that $\Div(yX_u)=0,$ where $X_u=\frac{y\nabla u}{\sqrt{1+y^2\td{\nabla u}^2}}$ is a vector field on $\Omega$, $\td{X_u}< 1$.

Denote by $\gamma$ an arc in $\overline{\Omega}\cap\hh^2$ such that its Euclidean length $\ell_\euc(\gamma)$ is finite. Denote by $\nu$ a unit normal to $\gamma$ in $\hh^2$. Then, we define the flux $F_u(\gamma)$ of $u$ across $\gamma$ by
$$F_u(\gamma)=\int_{\gamma}\vh{yX_u,\nu}\,\dsf s,$$
if $\gamma\subset\Omega$, if not, we define $F_u(\gamma)=F_u(\Gamma),$ where $\Gamma$ is an arc in $\Omega$ joining the end-points of $\gamma$ such that $\ell_\euc(\Gamma)<\infty.$ Clearly, $F_u(\gamma)$ changes sign if we choose $-\nu$ in place of $\nu$. In the case $\gamma\subset\dhr\Omega$, $\nu$ will always be chosen to be the outer normal to $\dhr\Omega$.

\begin{md}\label{lem2.5}
Let $u$ be a minimal graph on a domain $\Omega\subset\hh^2$.
\begin{itemize}
\item[\rm (i)] For every curve $\gamma$ in $\overline{\Omega}$ that $\ell_\euc(\gamma)<\infty$ we have $|F_u(\gamma)|\le \ell_\euc(\gamma).$
\item[\rm (ii)] For every admissible domain $\Omega'$ of $\Omega$ such that $\ell_\euc(\dhr\Omega')<\infty$, we have $F_u(\dhr\Omega')=0.$
\item[\rm (iii)] Let $\gamma$ be a curve in $\Omega$ or an Euclidean mean convex curve in $\dhr\Omega$ on which $u$ is continuous, obtains the finite value and $\ell_\euc(\gamma)<\infty$.
Then $F_u(\gamma)<\ell_\euc(\gamma)$.
\item[\rm (iv)] Let $\gamma\subset\dhr\Omega$ be an Euclidean geodesic arc such that $u$ diverges to $+\infty$ (resp. $-\infty$) as one approaches $\gamma$ within $\Omega$, then $F_u(\gamma)=\ell_\euc(\gamma)$ (resp. $F_u(\gamma)=-\ell_\euc(\gamma)$).
\end{itemize}
\end{md}
\begin{proof}
(i) - Case $\gamma\subset\Omega$. Since $\td{X_u}< 1$ we have
$$\left|F_u(\gamma)\right| \le \int_{\gamma}\left|\vh{yX_u,\nu}\right|\,\dsf s\le \int_{\gamma}y\,\dsf s=\ell_\euc(\gamma).$$
- Case $\gamma\not\subset\Omega$. For every positive real number $\varepsilon$, there is a curve $\Gamma\subset\Omega$ such that $\ell_\euc(\Gamma)\le\ell_\euc(\gamma)+\varepsilon$ and $F_u(\gamma)=F_u(\Gamma)$. Then,
$$|F_u(\gamma)|=|F_u(\Gamma)|\le\ell_\euc(\gamma)+\varepsilon.$$
This proved the result.

(ii) - Case $\Omega'$ is bounded. By divergence theorem, we have
$$F_u(\dhr\Omega')=\int_{\dhr\Omega'}\vh{yX_u,\nu}\,\dsf s=\int_{\Omega'}\Div(yX_u)\,\dsf \Acal=0.$$
- Case $\Omega'$ is unbounded. Denote by $E$ the set of ideal vertices of $\Omega'$. For each $p\in E$, we take a net of the geodesics $H_{p,n}$ that converges to $p$. (See Figure \ref{F9}). Let us denote by $\Hcal_{p,n}$ a domain of $\hh^2$ delimited by $H_{p,n}$ such that the Euclidean mean convex vector of $H_{p,n}$ pointing interior. We define $$\Omega'_n=\Omega'\setminus\bigcup_{p\in E}\overline{\Hcal}_{p,n}.$$


\begin{figure}[!h]
\centering
\definecolor{qqqqff}{rgb}{0,0,1}
\begin{tikzpicture}[line cap=round,line join=round,>=triangle 45,x=1.0cm,y=1.0cm]
\clip(-3.42,-2.16) rectangle (7.28,5.24);
\draw (1,-1)-- (-1,1);
\draw [shift={(3.66,0.73)}] plot[domain=-1.37:1.4,variable=\t]({1*1.76*cos(\t r)+0*1.76*sin(\t r)},{0*1.76*cos(\t r)+1*1.76*sin(\t r)});
\draw [shift={(2.38,2.17)}] plot[domain=0.18:2.6,variable=\t]({1*1.61*cos(\t r)+0*1.61*sin(\t r)},{0*1.61*cos(\t r)+1*1.61*sin(\t r)});
\draw [shift={(0.54,1)}] plot[domain=4.94:6.06,variable=\t]({1*2.05*cos(\t r)+0*2.05*sin(\t r)},{0*2.05*cos(\t r)+1*2.05*sin(\t r)});
\draw [shift={(-0.23,2.17)}] plot[domain=-0.35:0.59,variable=\t]({1*1.48*cos(\t r)+0*1.48*sin(\t r)},{0*1.48*cos(\t r)+1*1.48*sin(\t r)});
\draw (1.16,1.66)-- (2.54,2.52);
\draw [shift={(4.63,3.53)}] plot[domain=4.34:4.8,variable=\t]({1*2.52*cos(\t r)+0*2.52*sin(\t r)},{0*2.52*cos(\t r)+1*2.52*sin(\t r)});
\draw [shift={(5.3,0.12)}] plot[domain=2.02:3.08,variable=\t]({1*1*cos(\t r)+0*1*sin(\t r)},{0*1*cos(\t r)+1*1*sin(\t r)});
\draw (1,-1)-- (1.82,0.48);
\draw [shift={(-0.69,-0.28)}] plot[domain=-0.4:0.66,variable=\t]({1*1.83*cos(\t r)+0*1.83*sin(\t r)},{0*1.83*cos(\t r)+1*1.83*sin(\t r)});
\draw [->,dash pattern=on 4pt off 4pt] (-3,-1) -- (7,-1);
\draw (5.42,4.76) node[anchor=north west] {$\mathbb{H}^2$};
\draw (2.48,1.98) node[anchor=north west] {$\Omega'$};
\draw [dash pattern=on 2pt off 2pt] (1,-1) circle (0.76cm);
\draw [dash pattern=on 2pt off 2pt] (4,-1) circle (0.76cm);
\draw (-0.5,-0.32) node[anchor=north west] {$H_{p,n}$};
\draw [shift={(0.49,1.51)}] plot[domain=1.24:3.47,variable=\t]({1*1.58*cos(\t r)+0*1.58*sin(\t r)},{0*1.58*cos(\t r)+1*1.58*sin(\t r)});
\draw [shift={(2.14,3.95)}] plot[domain=3.83:4.99,variable=\t]({1*1.48*cos(\t r)+0*1.48*sin(\t r)},{0*1.48*cos(\t r)+1*1.48*sin(\t r)});
\draw [shift={(2.98,0.08)}] plot[domain=0.07:0.98,variable=\t]({1*1.32*cos(\t r)+0*1.32*sin(\t r)},{0*1.32*cos(\t r)+1*1.32*sin(\t r)});
\draw [shift={(4.56,0.99)}] plot[domain=3.36:4.44,variable=\t]({1*2.07*cos(\t r)+0*2.07*sin(\t r)},{0*2.07*cos(\t r)+1*2.07*sin(\t r)});
\draw [shift={(1.65,1.73)}] plot[domain=3.93:4.84,variable=\t]({1*1.27*cos(\t r)+0*1.27*sin(\t r)},{0*1.27*cos(\t r)+1*1.27*sin(\t r)});
\draw [->,dash pattern=on 4pt off 4pt] (-2,-2) -- (-2,5);
\draw (6.66,-0.36) node[anchor=north west] {$x$};
\draw (-1.76,5.16) node[anchor=north west] {$y$};
\draw (1,-1) node[anchor=north west] {$p$};
\begin{scriptsize}
\draw [fill=qqqqff] (1,-1) circle (1.5pt);
\draw [fill=qqqqff] (-1,1) circle (1.5pt);
\draw [fill=qqqqff] (1,3) circle (1.5pt);
\draw [fill=qqqqff] (4,-1) circle (1.5pt);
\draw [fill=qqqqff] (3.96,2.46) circle (1.5pt);
\draw [fill=qqqqff] (2.54,0.54) circle (1.5pt);
\draw [fill=qqqqff] (1.16,1.66) circle (1.5pt);
\draw [fill=qqqqff] (2.54,2.52) circle (1.5pt);
\draw [fill=qqqqff] (3.72,1.18) circle (1.5pt);
\draw [fill=qqqqff] (4.86,1.02) circle (1.5pt);
\draw [fill=qqqqff] (4.3,0.18) circle (1.5pt);
\draw [fill=qqqqff] (1.82,0.48) circle (1.5pt);
\draw [fill=qqqqff] (0.76,0.84) circle (1.5pt);
\draw [fill=black] (-2,-1) circle (1.0pt);
\draw[color=black] (-1.86,-1.2) node {$O$};
\end{scriptsize}
\end{tikzpicture}
\caption{The domain $\Omega'$ and $H_{p,n}$.}\label{F9}
\end{figure}
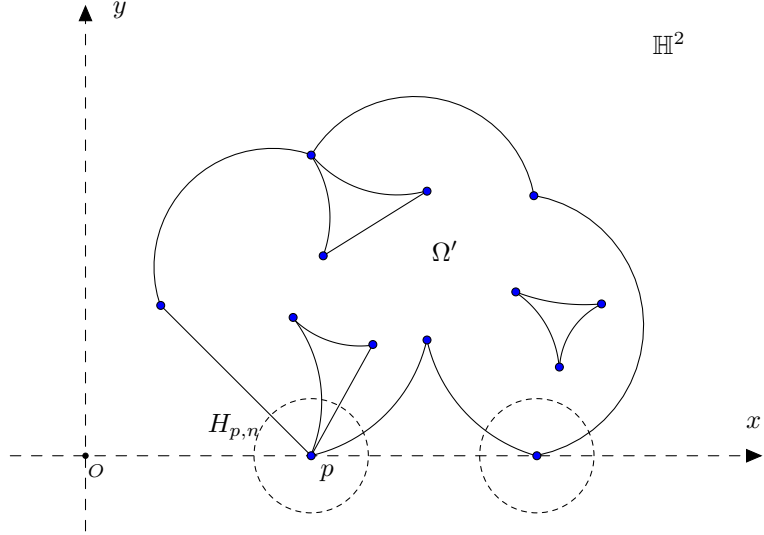

These subdomains of $\Omega'$ are bounded. It follows from the previous case that $F_u(\dhr\Omega'_n)=0$. Thus we have
$$F_u(\dhr\Omega')=F_u(\dhr\Omega')-F_u(\dhr\Omega'_n)=\sum_{p\in E}F_u(\dhr\Omega'\cap\Hcal_{p,n})-F_u(\dhr\Omega'_n\setminus\dhr\Omega').$$
Since $\ell_\euc(\dhr\Omega')<\infty$, we have
$$\sum_{p\in E}|F_u(\dhr\Omega'\cap\Hcal_{p,n})|\le \sum_{p\in E}\ell_\euc(\dhr\Omega'\cap\Hcal_{p,n})\to 0\quad\text{as }n\to\infty.$$
Moreover
$$|F_u(\dhr\Omega'_n\setminus\dhr\Omega')|\le \ell_\euc(\dhr\Omega'_n\setminus\dhr\Omega')\le \sum_{p\in E}\ell_\euc(H_{p,n})\to 0\quad\text{as }n\to\infty.$$
This completes the proof.

(iii) It is sufficient to show that $F_u(\gamma)< \ell_\euc(\gamma)$ for a small arc $\gamma$ .
Let $p\in\gamma$, there exists a positive $\varepsilon$ such that $\overline{\Dbb}_\varepsilon(p)\cap(\dhr\Omega\setminus\gamma)=\emptyset$. Let
$\Omega_\varepsilon(p):=\Omega\cap \Dbb_\varepsilon(p).$(See Figure \ref{F10}).
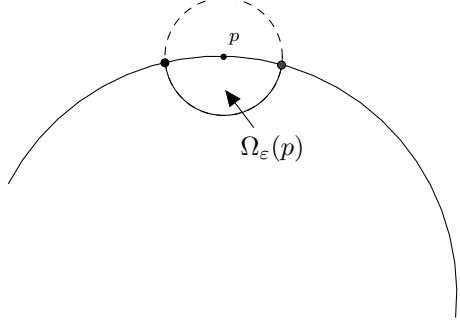
\begin{figure}[!h]
\centering
\definecolor{uuuuuu}{rgb}{0.27,0.27,0.27}
\begin{tikzpicture}[line cap=round,line join=round,>=triangle 45,x=1.0cm,y=1.0cm]
\clip(-1.64,1.02) rectangle (5.38,6.02);
\draw [shift={(1.72,1.91)}] plot[domain=-0.1:2.66,variable=\t]({1*3.16*cos(\t r)+0*3.16*sin(\t r)},{0*3.16*cos(\t r)+1*3.16*sin(\t r)});
\draw [dash pattern=on 3pt off 3pt] (1.78,5.06) circle (0.78cm);
\draw [shift={(1.78,5.06)}] plot[domain=3.25:6.14,variable=\t]({1*0.78*cos(\t r)+0*0.78*sin(\t r)},{0*0.78*cos(\t r)+1*0.78*sin(\t r)});
\draw (1.88,4.14) node[anchor=north west] {$\Omega_\varepsilon(p)$};
\draw [->] (2.18,4.12) -- (1.8,4.62);
\begin{scriptsize}
\draw [fill=black] (1.78,5.06) circle (1.0pt);
\draw[color=black] (1.92,5.3) node {$p$};
\draw [fill=black] (1,4.98) circle (1.5pt);
\draw [fill=uuuuuu] (2.55,4.95) circle (1.5pt);
\end{scriptsize}
\end{tikzpicture}
\caption{The domain $\Omega_\varepsilon(p)$.}\label{F10}
\end{figure}


By the general existence theorem, there is a minimal solution $v$ on $\Omega_\varepsilon(p)$ with $v=u+1$ on $\gamma$ and $v= u$ on $\dhr\Omega_\varepsilon(p)\setminus\gamma$.

It follows from the lemma \ref{lem cle}, that $$\int_{\Omega_\varepsilon(p)}\vh{\nabla v-\nabla u,yX_v-yX_u}\,\dsf \Acal>0.$$
Since $u,v$ are the minimal solutions
$$\vh{\nabla v-\nabla u,yX_v-yX_u}=\Div\left((v-w)(yX_v-yX_u)\right).$$
By the divergence theorem, we have
$$0<\int_{\dhr\Omega_\varepsilon(p)}\vh{(v-u)(yX_v-yX_u),\nu}\,\dsf s=F_v(\gamma)-F_u(\gamma). $$
Therefore
$$F_u(\gamma)< F_v(\gamma)\le\ell_\euc(\gamma),$$
which completes the proof.

(iv) We show for the case $u$ diverges to $+\infty$ as one approaches $\gamma$ within $\Omega$.
Without loss of generality, we assume that $\gamma$ is compact. We first prove that
\begin{equation}\label{equ:pt}
\lim_{q\to p}N_{u-u(q)}(q)\to -\nu(p),\qquad \forall\,p\in\gamma.
\end{equation}
Assume the contrary that there exists a sequence $q_n\in\Omega, q_n\to p$ such that $\lim_{n\to\infty}N_{u-u(q_n)}(q_n)=v\ne -\nu (p)$.
 Since $u_{|\gamma}=+\infty$,  there exists $R>0$ satisfies the distance $d_\Sigma(Q_n, \dhr \Sigma)>R$, $\forall n$ where $\Sigma=\Gr(u)$ and $ Q_n=(q_n,u(q_n) )$. Since $\Sigma$ is stable, we deduce from Schoen's curvature estimate \cite{Sch83} or \cite[Theorem 2.10]{CM11} that $$\td{A(q)}\le\kappa\qquad \forall\, q\in \Dbb^\Sigma_{R/2}\Ngoac{Q_n}$$ where $A$ is the second fundamental form of $\Sigma$ and $\kappa$ is an absolute constant.

Hence, by \cite[Lemma 2.4]{CM11},  around each $Q_n$ the surface $\Sigma$ is a graph over a disk $\Dbb_r(Q_n)$ of the tangent plane at $Q_n$ of $\Sigma$ and the graph has bounded distance from the disk $\Dbb_r(Q_n)$. The radius of the disk depends only on $R$, hence it is independent of $n$. So, if $q_n$ is close enough to $\gamma$, then the horizontal projection of $\Dbb_r(Q_n)$ and thus of the surface $\Sigma$ is not contained in $\Omega$, contradiction.

For $\delta>0$ sufficiently small, we take  $\Omega_\delta\subset\{q\in\Omega:d_\Omega(q,\gamma)<\delta \}$ such that for each point $q\in\Omega_\delta$, there will exist a unique point $p=p(q)\in\gamma$ such that $d_\Omega(q,p)=d_\Omega(q,\gamma)$. By using the diffeomorphism $\Omega_\delta\to \gamma\times[0,\delta),\quad q\mapsto (p(q),d(q,\gamma))$, we can extend $\nu$ on $\Omega_\delta$. By (\ref{equ:pt}), 
\begin{equation}\label{equ:pt1}
\lim_{q\in\Omega,q\to\gamma}\vh{X_u(q),\nu(q)}=1.
\end{equation}
Denote by $p_1,p_2$ two end-points of $\gamma$. Define $\gamma_\varepsilon=\{q\in\Omega_\delta:d(q,\gamma)=\varepsilon \}$ for every $0<\varepsilon<\delta$. Denote by $q_1,q_2$ two end-points of $\gamma_\varepsilon$ such that $d(p_i,q_i)=\varepsilon,i=1,2$. We have
$$F_u(\gamma)=F_u\Ngoac{\overline{p_1q_1}}+F_u(\gamma_\varepsilon)+F_u\Ngoac{\overline{q_2p_2}}.$$
Since $$F_u\Ngoac{\overline{p_iq_i}}\ge -\ell_\euc\Ngoac{\overline{p_iq_i}}\to 0,\qquad \text{as }\varepsilon\to 0, i=1,2,$$
and by (\ref{equ:pt1}) $$F_u(\gamma_\varepsilon)=\int_{\gamma_\varepsilon}\vh{X_u(q),\nu(q)}\,\dsf s_{\euc}(q)\to \ell_{\euc}(\gamma)\qquad\text{as }\varepsilon\to 0,$$
then $F_u(\gamma)\ge \ell_{\euc}(\gamma)$. Therefore $F_u(\gamma)= \ell_{\euc}(\gamma).$
 This completes the proof.
\end{proof}

\begin{md}\label{lem2.7} Let $\{u_n\}$ be a sequence of minimal graphs on a fixed domain $\Omega\subset\hh$ which extends continuously to $\dhr\Omega$ and let $A$ be an Euclidean geodesic arc in $\dhr\Omega$ such that $\ell_\euc(A)<\infty$. Then
\begin{itemize}
\item[\rm (i)] If $\{u_n\}$ diverges uniformly to $+\infty$ on compact sets of $A$ and while remaining uniformly bounded on compact sets of $\Omega$, then
$$\lim_{n\to\infty}F_{u_n}(A)=\ell_\euc(A).$$
\item[\rm (ii)] If $\{u_n\}$ diverges uniformly to $+\infty$ on compact sets of $\Omega$ while remaining uniformly bounded on compact sets of and $A$ , then
$$\lim_{n\to\infty}F_{u_n}(A)=-\ell_\euc(A).$$
\end{itemize}
\end{md}


\section{Monotone convergence theorem and Divergence set theorem}\label{sec5}

\subsection{Monotone convergence theorem}

\begin{dl}[Local Harnack inequality]\label{IHL}
Let $u$ be a nonnegative minimal solution on $\Omega=\Dbb_R(P)\subset\hh^2$ and let $Q$ be a point of $\Omega$. There is a function $\Phi(t,r)$ ( that does not depend on the function $u$) such that
$$|u(Q)|\le \Phi\left(|u(P)|,d(P,Q)\right).$$
For each fixed $t$, $\Phi(t,-)$ is a continuous strictly increasing function defined on a interval $[0,r(t))$ with
$$\Phi(t,0)=t,\qquad \lim_{r\to r(t)^-}\Phi(t,r)=\infty$$
where $t\mapsto r(t)$ is a continuous strictly decresing function tending to zero as $t$ tends to infinity.
\end{dl}
\begin{proof} Let $\gamma:[0,R)\hookrightarrow \Omega$ be an geodesic arc that satisfies
$$\gamma(0)=P,\qquad \td{\gamma'}=1,\qquad Q\in\gamma([0,R)) .$$
Define $\hat{u}: [0,R)\to\rr$ by condition
$$\hat{u}(r)=u(\gamma(r)).$$
By Theorem \ref{EGI},
\begin{equation}
\hat{u}'(r)\le f\left(\dfrac{\hat{u}(r)}{R-r}\right).
\end{equation}
For each $t>0,$ we define a function $r\mapsto\Phi(t,r)$ by the conditions
\begin{equation}
\dfrac{d\Phi}{d r}(t,r)=f\left(\dfrac{\Phi(t,r)}{R-r}\right),\qquad \Phi(t,0)=t.
\end{equation}
Then $\hat{u}(r)\le \Phi(\hat{u}(0),r)$ whenever $\Phi$ is well defined.
\end{proof}

\begin{dl}[Dini's monotone convergence theorem]
If $X$ is a compact topological space, and $\{f_n\}$ is a monotonically increasing sequence (meaning $f_n(x)\le f_{n+1}(x)$ for all $n$ and $x$) of continuous real-valued functions on $X$ which converges pointwise to a continuous function $f$, then the convergence is uniform. The same conclusion holds if $\{f_n\}$ is monotonically decreasing instead of increasing.
\end{dl}

\begin{dl}[Monotone convergence theorem]\label{TCM}
Let $\{u_n\}$ be a monotone increasing  sequence of minimal graphs on a domain $\Omega\subset\hh^2$. There exists an open set $\Ucal\subset\Omega$ (called the convergence set) such that $\{u_n\}$ converges uniformly on compact subset of $\Ucal$ and diverges uniformly to $+\infty$ on compact subsets of $\Vcal:=\Omega\setminus \Ucal$ (divergence set). Moreover, if $u_n$ is bounded at a point $p\in\Omega$, then the convergence set $\Ucal$ is non-empty (it contains a neighborhood of $p$).
\end{dl}
\begin{proof}
Let $\{u_n\}$ be an increasing sequence. 
Denote by $P$ a point in $\Ucal: =\left\{x\in\Omega:\sup_{n\ge 0}|u_n(x)|<\infty\right\}$. There is a positive number $R$ such that
$$\Dbb_R(P)\subset\Omega,\qquad C:=\inf_{x\in \Dbb_R(P)}u_1(x)>-\infty.$$
Let $m:=-C+\sup_{n\ge 0}u_n(P)$. The function $\Phi$ is well defined on the interval $[0,r(m))$. Define $\varepsilon:=\min\left\{\dfrac{r(m)}{2},R\right\}$. For each $Q\in \Dbb_\varepsilon(P)$, by using the local Harnach inequality, we have
$$-C+u_n(Q)\le \Phi(-C+u_n(P),d(P,Q))\le \Phi\left(m,\frac{r(m)}{2}\right).$$
By definition, $\Dbb_\varepsilon(P)\subset U$. Then $\Ucal$ is open.
\end{proof}

\subsection{Divergence set theorem}
\begin{dl}[Divergence set theorem $\Vcal$]\label{SD}
Let $\Omega\subset\hh^2$ be a admissible domaine whose boundary is composed with finitely Euclidean mean convex arcs $C_i$. Let $\{u_n\}$ be an increasing or decreasing sequence of minimal graphs on $\Omega$, respectively. Then, for each open arc $C_i$, we assume that, for every $n$, $u_n$ extends continuously on $C_i$ and either $\{u_n|_{C_i}\}$ converges to a continuous function or $\infty$ or $-\infty$, respectively. Let $\Vcal=\Vcal(\{u_n\})$ be the divergence set associated to $\{u_n\}$.
\begin{itemize}
\item[\rm (i)] The boundary of $\Vcal$ consists of the union of a set of non-intersecting interior Euclidean geodesic chords in $\Omega$ joining two points of $\dhr\Omega$, together with arcs in $\dhr\Omega$. Moreover, a component of $\Vcal$ cannot be an isolated point.
\item[\rm (ii)] A component of $\Vcal$ cannot be an interior chord.
\item[\rm (iii)]No two interior chords in $\dhr\Vcal$ can have a common endpoint at a convex corner of $\Vcal$.
\item[\rm (iv)] The endpoints of interior Euclidean geodesic chords are among the vertices of $\dhr\Omega$. So the boundary of $\Vcal$ has a finite set of interior Euclidean geodesic chords in $\Omega$ joining two vertices of $\dhr\Omega$.
\end{itemize}
\end{dl}

\begin{proof}
Without loss of generality, assume that the sequence $\{u_n\}$ is increasing and the divergence set is not empty.

(i) It is clear by Lemma \ref{lem-droite0} and Corollary \ref{lem-droite} that each arc of $\dhr\Vcal$ must be Euclidean geodesic and that no vertex of $\dhr\Vcal$ lies in $\Omega$, then (i) follows.


(iii) Assume the contrary that (iii) does not hold. Let $\gamma_1,\gamma_2$ be two arcs of $\dhr\Vcal$ having a common endpoint $p\in\dhr\Vcal$ at a convex corner. Choose two points $q_i\in\gamma_i,i=1,2$ such that the triangle $\triangle$ with vertices $p,q_1,q_2$ lies in $\Omega$. We can always assume that the triangle $\triangle$ is either in $\Ucal$ or in $\Vcal$. Indeed, if $\triangle\not\subset\Vcal$, we take a component $\triangle'$ of $\Ucal\cap\triangle$. Let $\gamma'_1, \gamma'_2$ be two Euclidean geodesic chords in $\Omega$ having a common endpoint $p$ such that the domain delimited by them is the smallest domain containing $\triangle'$. Then $\gamma'_1,\gamma'_2\subset\dhr\Vcal$ and $\triangle' $ is  the triangle delimited by $\gamma'_1,\gamma'_2$ and $\overline{q_1q_2}$ and $\triangle'\subset\Ucal$. We can choose $\gamma'_1,\gamma'_2$ in place of $\gamma_1,\gamma_2$.
By the lemma  \ref{lem2.7}, 
$$0=F_{u_n}(\dhr\triangle)=F_{u_n}(\overline{pq_1})+F_{u_n}(\overline{pq_2})+F_{u_n}(\overline{q_1q_2}),$$
$$\lim_{n\to\infty}F_{u_n}(\overline{pq_i})=\begin{cases}
\ell_\euc(\overline{pq_i})&\text{ if }\triangle\subset\Ucal\\
-\ell_\euc(\overline{pq_i})&\text{ if }\triangle\subset\Vcal
\end{cases}\qquad i=1,2.$$
On the other hand $ \lim_{n\to\infty}\left|F_{u_n}(\overline{q_1q_2})\right|\le \ell_\euc(\overline{q_1q_2}).$ Hence
$$\ell_\euc(\overline{q_1q_2})\ge \ell_\euc(\overline{pq_1})+\ell_\euc(\overline{pq_2}),$$
a contradiction.

(ii) and (iv) are proved with analogous arguments, using lemma \ref{lem-droite0} and corollary \ref{lem-droite}. The details are left to the reader.
\end{proof}


\section{Jenkins-Serrin type theorem}\label{sec6}

Let $\Omega\subset\hh^2$ be a domain whose boundary $\dhr_\infty\Omega$ consists of a finite number of Euclidean geodesic arcs $A_i,B_i$, a finite number of Euclidean mean convex arcs $C_i$ (towards $\Omega$) together with their endpoints, which are called the vertices of $\Omega$. We mark the
$A_i$ edges by $+\infty$ and the $B_i$ edges by $-\infty$, and assign arbitrary continuous data $f_i$ on the arcs $C_i$, respectively. 
Assume that no two $A_i$ edges and no two $B_i$ edges meet at a convex corner. 
We call such a domain $\Omega$   \textit{Scherk domain}. (See Figure \ref{F11}.)
Assume in addition that, the vertices at infinity of Scherk domain are the removable points at infinity.


\begin{figure}[!h]
\centering
\definecolor{qqqqff}{rgb}{0,0,1}
\begin{tikzpicture}[line cap=round,line join=round,>=triangle 45,x=1.0cm,y=1.0cm]
\clip(-2,-1.4) rectangle (8.2,5.64);
\draw (2.06,-0.56)-- (0.06,1.44);
\draw [shift={(1.86,1.64)}] plot[domain=1.46:3.25,variable=\t]({1*1.81*cos(\t r)+0*1.81*sin(\t r)},{0*1.81*cos(\t r)+1*1.81*sin(\t r)});
\draw [shift={(4.14,1.89)}] plot[domain=-1.21:1.38,variable=\t]({1*2.62*cos(\t r)+0*2.62*sin(\t r)},{0*2.62*cos(\t r)+1*2.62*sin(\t r)});
\draw [shift={(1.6,1.44)}] plot[domain=4.94:6.06,variable=\t]({1*2.05*cos(\t r)+0*2.05*sin(\t r)},{0*2.05*cos(\t r)+1*2.05*sin(\t r)});
\draw [shift={(0.97,2.63)}] plot[domain=-0.4:0.64,variable=\t]({1*1.35*cos(\t r)+0*1.35*sin(\t r)},{0*1.35*cos(\t r)+1*1.35*sin(\t r)});
\draw [shift={(3.57,5.48)}] plot[domain=4.07:4.74,variable=\t]({1*2.54*cos(\t r)+0*2.54*sin(\t r)},{0*2.54*cos(\t r)+1*2.54*sin(\t r)});
\draw (2.22,2.1)-- (3.64,2.94);
\draw [shift={(6.36,0.56)}] plot[domain=2.02:3.08,variable=\t]({1*1*cos(\t r)+0*1*sin(\t r)},{0*1*cos(\t r)+1*1*sin(\t r)});
\draw (2.06,-0.56)-- (2.88,0.92);
\draw [shift={(0.37,0.16)}] plot[domain=-0.4:0.66,variable=\t]({1*1.83*cos(\t r)+0*1.83*sin(\t r)},{0*1.83*cos(\t r)+1*1.83*sin(\t r)});
\draw [shift={(2.79,2.4)}] plot[domain=4:4.77,variable=\t]({1*1.49*cos(\t r)+0*1.49*sin(\t r)},{0*1.49*cos(\t r)+1*1.49*sin(\t r)});
\draw [->,dash pattern=on 4pt off 4pt] (-1.94,-0.56) -- (8.06,-0.56);
\draw (6.48,5.2) node[anchor=north west] {$\mathbb{H}^2$};
\draw (1.1,2.3) node[anchor=north west] {$\Omega$};
\draw (0.42,0.74) node[anchor=north west] {$A_1$};
\draw (0.12,3.62) node[anchor=north west] {$C_1$};
\draw (1.8,3.28) node[anchor=north west] {$C_2$};
\draw (2.68,3.64) node[anchor=north west] {$C_3$};
\draw (6.58,1.5) node[anchor=north west] {$C_4$};
\draw (2.06,3.44)-- (4.64,4.46);
\draw (3.04,4.58) node[anchor=north west] {$B_1$};
\draw (2.94,2.72) node[anchor=north west] {$A_2$};
\draw (3.6,0.98)-- (5.06,-0.56);
\draw (4.78,1.62)-- (5.36,0.62);
\draw (4.78,1.62)-- (5.92,1.46);
\draw (5.06,2.14) node[anchor=north west] {$A_3$};
\draw (4.48,1.5) node[anchor=north west] {$A_4$};
\draw (5.46,1.28) node[anchor=north west] {$C_5$};
\draw (4.36,0.66) node[anchor=north west] {$B_2$};
\draw (3.32,0.48) node[anchor=north west] {$C_6$};
\draw (1.54,0.86) node[anchor=north west] {$C_7$};
\draw (2.3,1.52) node[anchor=north west] {$C_8$};
\draw (2.5,0.74) node[anchor=north west] {$B_3$};
\draw (4.08,2.58) node[anchor=north west] {$\Omega$};
\draw [->,dash pattern=on 4pt off 4pt] (-0.92,-1.28) -- (-0.94,5.44);
\draw (7.84,0.04) node[anchor=north west] {$x$};
\draw (-0.88,5.62) node[anchor=north west] {$y$};
\begin{scriptsize}
\draw [fill=qqqqff] (2.06,-0.56) circle (1.5pt);
\draw [fill=qqqqff] (0.06,1.44) circle (1.5pt);
\draw [fill=qqqqff] (2.06,3.44) circle (1.5pt);
\draw [fill=qqqqff] (5.06,-0.56) circle (1.5pt);
\draw [fill=qqqqff] (4.64,4.46) circle (1.5pt);
\draw [fill=qqqqff] (3.6,0.98) circle (1.5pt);
\draw [fill=qqqqff] (2.22,2.1) circle (1.5pt);
\draw [fill=qqqqff] (3.64,2.94) circle (1.5pt);
\draw [fill=qqqqff] (4.78,1.62) circle (1.5pt);
\draw [fill=qqqqff] (5.92,1.46) circle (1.5pt);
\draw [fill=qqqqff] (5.36,0.62) circle (1.5pt);
\draw [fill=qqqqff] (2.88,0.92) circle (1.5pt);
\draw [fill=qqqqff] (1.82,1.28) circle (1.5pt);
\draw [fill=black] (-0.94,-0.56) circle (1.5pt);
\draw[color=black] (-0.8,-0.28) node {$O$};
\end{scriptsize}
\end{tikzpicture}
\caption{An example of Scherk domain} \label{F11}
\end{figure}
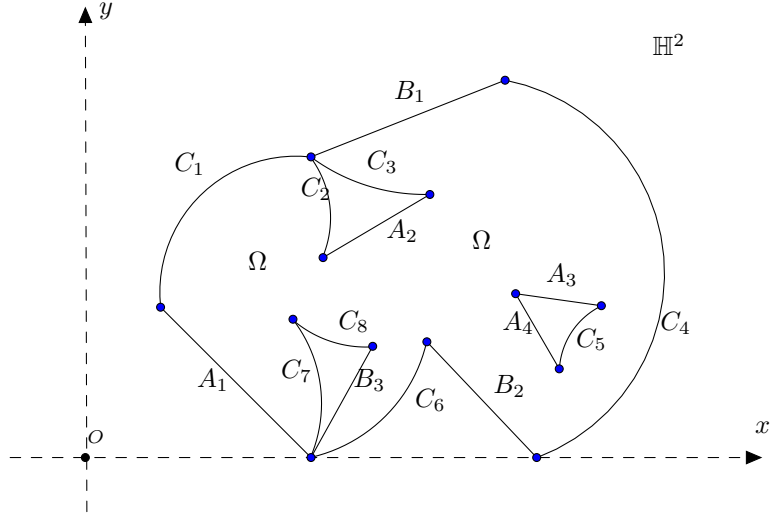


An \textit{Euclidean polygonal domain} $\Pcal$ in $\hh^2$ is a domain whose boundary $\dhr_\infty\Pcal$ is composed of finitely many Euclidean geodesic arcs in $\hh^2$ together with their endpoints, which are called the vertices of $\Pcal$.

An Euclidean polygonal domain $\Pcal$ is said to be inscribed in a Scherk domain $\Omega$ if $\Pcal\subset\Omega$ and its vertices
are among the vertices of $\Omega$. We notice that a vertex may be in $\dhr_\infty\Omega$
and an edge may be one of the $A_i$ or $B_i$. (See Figure \ref{F12}).

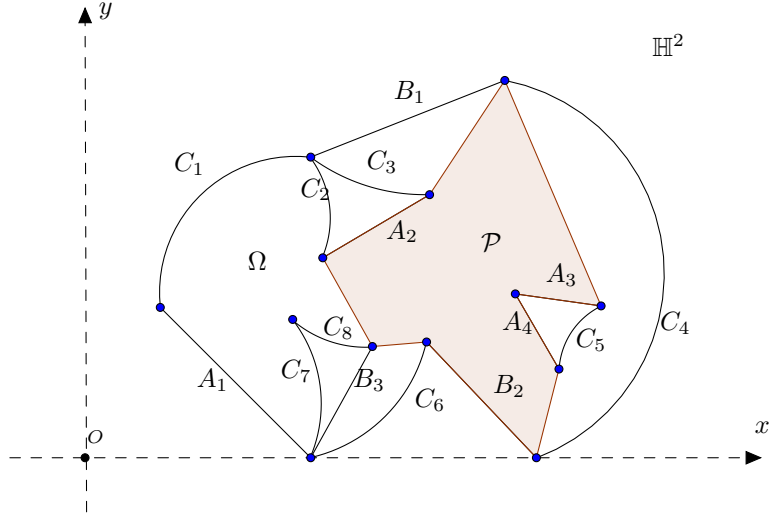
\begin{figure}[!h]

\centering

\definecolor{zzttqq}{rgb}{0.6,0.2,0}
\definecolor{qqqqff}{rgb}{0,0,1}
\begin{tikzpicture}[line cap=round,line join=round,>=triangle 45,x=1.0cm,y=1.0cm]
\clip(-2,-1.82) rectangle (8.14,5.24);
\fill[color=zzttqq,fill=zzttqq,fill opacity=0.1] (3.54,0.54) -- (5,-1) -- (5.3,0.18) -- (4.72,1.18) -- (5.86,1.02) -- (4.58,4.02) -- (3.58,2.5) -- (2.16,1.66) -- (2.82,0.48) -- cycle;
\draw (2,-1)-- (0,1);
\draw [shift={(1.8,1.2)}] plot[domain=1.46:3.25,variable=\t]({1*1.81*cos(\t r)+0*1.81*sin(\t r)},{0*1.81*cos(\t r)+1*1.81*sin(\t r)});
\draw [shift={(4.08,1.45)}] plot[domain=-1.21:1.38,variable=\t]({1*2.62*cos(\t r)+0*2.62*sin(\t r)},{0*2.62*cos(\t r)+1*2.62*sin(\t r)});
\draw [shift={(1.54,1)}] plot[domain=4.94:6.06,variable=\t]({1*2.05*cos(\t r)+0*2.05*sin(\t r)},{0*2.05*cos(\t r)+1*2.05*sin(\t r)});
\draw [shift={(0.91,2.19)}] plot[domain=-0.4:0.64,variable=\t]({1*1.35*cos(\t r)+0*1.35*sin(\t r)},{0*1.35*cos(\t r)+1*1.35*sin(\t r)});
\draw [shift={(3.51,5.04)}] plot[domain=4.07:4.74,variable=\t]({1*2.54*cos(\t r)+0*2.54*sin(\t r)},{0*2.54*cos(\t r)+1*2.54*sin(\t r)});
\draw (2.16,1.66)-- (3.58,2.5);
\draw [shift={(6.3,0.12)}] plot[domain=2.02:3.08,variable=\t]({1*1*cos(\t r)+0*1*sin(\t r)},{0*1*cos(\t r)+1*1*sin(\t r)});
\draw (2,-1)-- (2.82,0.48);
\draw [shift={(0.31,-0.28)}] plot[domain=-0.4:0.66,variable=\t]({1*1.83*cos(\t r)+0*1.83*sin(\t r)},{0*1.83*cos(\t r)+1*1.83*sin(\t r)});
\draw [shift={(2.73,1.96)}] plot[domain=4:4.77,variable=\t]({1*1.49*cos(\t r)+0*1.49*sin(\t r)},{0*1.49*cos(\t r)+1*1.49*sin(\t r)});
\draw [->,dash pattern=on 4pt off 4pt] (-2,-1) -- (8,-1);
\draw (6.42,4.76) node[anchor=north west] {$\mathbb{H}^2$};
\draw (1.04,1.86) node[anchor=north west] {$\Omega$};
\draw (0.36,0.3) node[anchor=north west] {$A_1$};
\draw (0.06,3.18) node[anchor=north west] {$C_1$};
\draw (1.74,2.84) node[anchor=north west] {$C_2$};
\draw (2.62,3.2) node[anchor=north west] {$C_3$};
\draw (6.52,1.06) node[anchor=north west] {$C_4$};
\draw (2,3)-- (4.58,4.02);
\draw (2.98,4.14) node[anchor=north west] {$B_1$};
\draw (2.88,2.28) node[anchor=north west] {$A_2$};
\draw (3.54,0.54)-- (5,-1);
\draw (4.72,1.18)-- (5.3,0.18);
\draw (4.72,1.18)-- (5.86,1.02);
\draw (5,1.7) node[anchor=north west] {$A_3$};
\draw (4.42,1.06) node[anchor=north west] {$A_4$};
\draw (5.4,0.84) node[anchor=north west] {$C_5$};
\draw (4.3,0.22) node[anchor=north west] {$B_2$};
\draw (3.26,0.04) node[anchor=north west] {$C_6$};
\draw (1.48,0.42) node[anchor=north west] {$C_7$};
\draw (2.04,1) node[anchor=north west] {$C_8$};
\draw (2.44,0.3) node[anchor=north west] {$B_3$};
\draw [color=zzttqq] (3.54,0.54)-- (5,-1);
\draw [color=zzttqq] (5,-1)-- (5.3,0.18);
\draw [color=zzttqq] (5.3,0.18)-- (4.72,1.18);
\draw [color=zzttqq] (4.72,1.18)-- (5.86,1.02);
\draw [color=zzttqq] (5.86,1.02)-- (4.58,4.02);
\draw [color=zzttqq] (4.58,4.02)-- (3.58,2.5);
\draw [color=zzttqq] (3.58,2.5)-- (2.16,1.66);
\draw [color=zzttqq] (2.16,1.66)-- (2.82,0.48);
\draw [color=zzttqq] (2.82,0.48)-- (3.54,0.54);
\draw (4.14,2.12) node[anchor=north west] {$\mathcal P$};
\draw [->,dash pattern=on 4pt off 4pt] (-0.98,-1.72) -- (-1,5);
\draw (7.78,-0.4) node[anchor=north west] {$x$};
\draw (-0.94,5.18) node[anchor=north west] {$y$};
\begin{scriptsize}
\draw [fill=qqqqff] (2,-1) circle (1.5pt);
\draw [fill=qqqqff] (0,1) circle (1.5pt);
\draw [fill=qqqqff] (2,3) circle (1.5pt);
\draw [fill=qqqqff] (5,-1) circle (1.5pt);
\draw [fill=qqqqff] (4.58,4.02) circle (1.5pt);
\draw [fill=qqqqff] (3.54,0.54) circle (1.5pt);
\draw [fill=qqqqff] (2.16,1.66) circle (1.5pt);
\draw [fill=qqqqff] (3.58,2.5) circle (1.5pt);
\draw [fill=qqqqff] (4.72,1.18) circle (1.5pt);
\draw [fill=qqqqff] (5.86,1.02) circle (1.5pt);
\draw [fill=qqqqff] (5.3,0.18) circle (1.5pt);
\draw [fill=qqqqff] (2.82,0.48) circle (1.5pt);
\draw [fill=qqqqff] (1.76,0.84) circle (1.5pt);
\draw [fill=black] (-1,-1) circle (1.5pt);
\draw[color=black] (-0.86,-0.72) node {$O$};
\end{scriptsize}
\end{tikzpicture}
\caption{ A polygonal domain $\Pcal$ inscribed in $\Omega$.}\label{F12}
\end{figure}

Given a polygonal domain $\Pcal$ inscribed in $\Omega$, we denote by $\ell_\euc(\Pcal)$ the Euclidean perimeter of $\dhr\Pcal$, and by $a_\euc(\Pcal)$ and $b_\euc(\Pcal)$ the total Euclidean lengths of the edges $A_i$ and $B_i$ lying in $\dhr\Pcal$, respectively.

Now is a good time to state and to prove the main theorem of this paper.
\begin{dl}\label{J-S type1}
Let $\Omega$ be a Scherk domain in $\hh^2$ with the families $\{A_i \},\{B_i \},\{C_i \}$.
\begin{itemize}
\item[\rm (i)] If the family $\{C_i\}$ is non-empty, there exists a solution to the Dirichlet problem on $\Omega$ if and only if
\begin{equation}\label{cond1}
2a_\euc(\Pcal)<\ell_\euc(\Pcal),\qquad 2b_\euc(\Pcal)<\ell_\euc(\Pcal)
\end{equation}
for every Euclidean polygonal domain inscribed in $\Omega$. Moreover, such a solution is unique if it exists.
\item[\rm (ii)] If the family $\{ C_i\}$ is empty, there exists a solution to the Dirichlet problem on $\Omega$ if and only if
\begin{equation}\label{cond2}
a_\euc(\Pcal)=b_\euc(\Pcal)
\end{equation}
when $\Pcal=\Omega$ and the inequalities in (\ref{cond1}) hold for all other Euclidean polygonal domains inscribed in $\Omega$. Such a solution is unique up to an additive constant, if it exists.\end{itemize}
\end{dl}
This theorem is similar in spirit to that of \cite{JS66,NR02,CR10,Pin09}.
\begin{proof}
\textit{The uniqueness} of the solution is deduced from Theorem \ref{PMNDM}.

Let us now prove that the conditions of theorem \ref{J-S type1} are necessary for the existence.
Assume that there is a minimal graph $u$ on $\Omega$ satisfying the Dirichlet problem.
When $\{C_i\}=\emptyset$ and $\Pcal=\Omega$, using the proposition \ref{lem2.5}, we have
\begin{align*}
0=F_u(\dhr\Pcal)&=\sum_{A_i\subset\dhr\Pcal}F_u(A_i)+\sum_{B_i\subset\dhr\Pcal}F_u(B_i)\\
&=\sum_{A_i\subset\dhr\Pcal}\ell_\euc(A_i)+\sum_{B_i\subset\dhr\Pcal}-\ell_\euc(B_i)=a_\euc(\Pcal)-b_\euc(\Pcal),
\end{align*}
as the condition (\ref{cond2}).

In the other case, $\dhr\Pcal\setminus\left(\underset{A_i\subset\dhr\Pcal}{\bigcup}A_i\cup \underset{B_i\subset\dhr\Pcal}{\bigcup}B_i \right)\ne\emptyset$ and $u$ is continuous on this set. By Proposition \ref{lem2.5}, we have

\begin{align*}
0&=F_u(\dhr\Pcal)\\
&=\sum_{A_i\subset\dhr\Pcal}F_u(A_i)+\sum_{B_i\subset\dhr\Pcal}F_u(B_i)+F_u\left(\dhr\Pcal\setminus\left(\underset{A_i\subset\dhr\Pcal}{\bigcup}A_i\cup \underset{B_i\subset\dhr\Pcal}{\bigcup}B_i \right) \right),
\end{align*}
\begin{gather*}
\sum_{A_i\subset\dhr\Pcal}F_u(A_i)=\sum_{A_i\subset\dhr\Pcal}\ell_\euc(A_i)=a_\euc(\Pcal),\\
\sum_{B_i\subset\dhr\Pcal}F_u(B_i)=\sum_{B_i\subset\dhr\Pcal}-\ell_\euc(B_i)=-b_\euc(\Pcal)
\end{gather*}
and
\begin{align*}
\left|F_u\left(\dhr\Pcal\setminus\left(\underset{A_i\subset\dhr\Pcal}{\bigcup}A_i\cup \underset{B_i\subset\dhr\Pcal}{\bigcup}B_i \right) \right)\right|&<\ell_\euc\left(\dhr\Pcal\setminus\left(\underset{A_i\subset\dhr\Pcal}{\bigcup}A_i\cup \underset{B_i\subset\dhr\Pcal}{\bigcup}B_i \right) \right)\\
&=\ell_\euc(\Pcal)-a_\euc(\Pcal)-b_\euc(\Pcal).
\end{align*}
We obtain $|a_\euc(\Pcal)-b_\euc(\Pcal)|<\ell_\euc(\Pcal)-a_\euc(\Pcal)-b_\euc(\Pcal)$. It follows the conditions (\ref{cond1}).

Finally, we prove that the conditions of theorem \ref{J-S type1} are sufficient. We distinguish the following cases:

\begin{trh}\label{cas2.1}
First case: Assume that the families $\{A_i\}$ and $\{B_i\}$ are both empty and the continuous functions $f_i$ are bounded.
\end{trh}
\begin{proof} For any ideal vertex $p$ of $\Omega$, we take a net of geodesics $H_{p,n}$ which converges to $p$. Denote by $\Hcal_{p,n}$ the domain of $\hh^2$ delimited by $H_{p,n}$ such that the Euclidean mean convex vector of $H_{p,n}$ points interior.
Let us define $\Omega_n$ an Euclidean convex subdomain of $\Omega$ delimited by $\dhr\Omega\setminus\bigcup_{i}\Hcal_{i,n}$ and by the Euclidean geodesics in $\Omega\cap\bigcup_{i}\Hcal_{i,n}$ joining the points of $\dhr\Omega\cap\bigcup_{i}\Hcal_{i,n}.$

By Theorem \ref{TEG}, for each positive natural number $n$, there exists a minimal solution $u_n$ on an Euclidean polygonal domain of $\Omega_n$ such that
$$u_n=\begin{cases}
f_i&\text{ on } C_i\cap\dhr\Omega_n,\\
0 &\text{ on the rest of } \dhr\Omega_n.
\end{cases}$$

By Maximum theorem, Theorem \ref{PMG}, the sequence $\{u_n\}$ is uniformly bounded on $\Omega$. By Compactness theorem, Theorem \ref{TC}, there exists a subsequence of the sequence $\{u_n\}_n$ converges uniformly on every compact set of $\Omega$ to a minimal solution $u:\Omega\to\rr$ that obtains the values $f_i$ on $C_i$.
\renewcommand{\qedsymbol}{$\heartsuit$}
\end{proof}

\begin{trh}\label{cas2.2}
Second case: The family $\{B_i \}$ is empty and the functions $f_i$ are non-negative.
\end{trh}
\begin{proof}

There exists, by the previous step \ref{cas2.1}, for each $n$, a minimal solution $u_n$ on $\Omega$ such that
$$u_n=\begin{cases}
n&\text{ on } \bigcup_iA_i\\
\min\{n,f_i\} & \text{ on } C_i.
\end{cases}$$
It follows from the maximum principle, Theorem \ref{PMG}, that $0\le u_n\le u_{n+1}$ for each $n$.

\begin{kd}
The divergence set $\Vcal=\Vcal(\{u_n\})$ is empty.
\end{kd}
\begin{proof}
Assume the contrary, that $\Vcal$ is not empty. By the lemma \ref{lem-droite} and Theorem \ref{SD}, $\Vcal$ consists of a finite number of Euclidean polygonal domains inscribed in $\Omega$. Let $\Pcal$ be a component of $\Vcal$. By Lemmas \ref{lem2.5} and \ref{lem2.7}, we have
$$0=F_{u_n}(\dhr\Pcal)=\sum_iF_{u_n}(A_i\cap\dhr\Pcal)+F_{u_n}\left(\dhr\Pcal\setminus\bigcup_iA_i\right),$$
$$\left|\sum_iF_{u_n}(A_i\cap\dhr\Pcal)\right|\le \sum_i\left|F_{u_n}(A_i\cap\dhr\Pcal)\right|\le \sum_i\ell_\euc(A_i)=a_\euc(\Pcal),$$
$$\lim_{n\to\infty}F_{u_n}\left(\dhr\Pcal\setminus\bigcup_iA_i\right)=-\ell_\euc\left(\dhr\Pcal\setminus\bigcup_iA_i\right)=-(\ell_\euc(\Pcal)-a_\euc(\Pcal)).$$
We conclude that $\ell_\euc(\Pcal)-a_\euc(\Pcal)\le a_\euc(\Pcal)$, which contradicts with the condition (\ref{cond1}).
\renewcommand{\qedsymbol}{$\Diamond$}
\end{proof}
By the previous assertion, we have $\Ucal(\{u_n\})=\Omega.$ Thus $\{u_n\}$ converges uniformly on the compact sets of $\Omega$ to a minimal solution $u$. By Theorem \ref{LVB}, $u$ takes the values $+\infty$ on $A_i$ and $f_i$ on $C_i$.
\renewcommand{\qedsymbol}{$\heartsuit$}
\end{proof}

\begin{trh}\label{cas2.3}
Third case: the family $\{C_i\}$ is non-empty.
\end{trh}
\begin{proof}
By the previous step, \ref{cas2.1} and \ref{cas2.2}, there exists the minimal solutions $u^{+},u^-$ and $u_n$ on $\Omega$ with the following boundary values
\begin{center}
\begin{tabular}{c|c|c|c}
& $A_i$ & $B_i$ & $C_i$ \\
\hline $u^+$ & $+\infty$ & 0 & $\max\{f_i,0\} $ \\
\hline $u_n$ & $n$ & $-n$ & $[f_i]_{-n}^n$ \\
\hline $u^-$ & $0$ & $-\infty$ & $\min\{f_i,0\}.$\\
\end{tabular}
\end{center}
It follows from Theorem \ref{PMNDM}, that $u^-\le u_n\le u^+$ for each $n$. By the compactness theorem, Theorem \ref{TC} and a diagonal process, we can extract a subsequence of $\{u_n\}$ which converges on compact sets of $\Omega$ to a minimal graph $u$. Moreover, by Theorem \ref{LVB}, $u$ takes the desired boundary conditions.
\renewcommand{\qedsymbol}{$\heartsuit$}
\end{proof}

\begin{trh}\label{cas2.4}
Fourth case: The family $\{C_i\}$ is empty.
\end{trh}
\begin{proof}
We fix a positive natural number $n$. There exists, by Case \ref{cas2.1}, a minimal solution $v_n$ on $\Omega$ that obtains the values $n$ on $A_i$ and $0$ on $B_i$. It follows from Theorem \ref{PMNDM}, that $0\le v_n\le n.$ For each $c\in (0,n)$, we define
$$E_c=\{v_n>c\},\qquad F_c=\{v_n<c\}.$$
Since $v_n=n$ on $A_i$, there exists a component $E^i_c$ of $E_c$ satifying $A_i\subset\dhr E^i_c$. Moreover, by the maximum principle, Theorem \ref{PMNDM}, $E_c=\bigcup_iE^i_c.$ Similarly, there exists, for each $i$, a component $F^i_c$ of $F_c$ satifying $B_i\subset\dhr F^i_c$, and, we have $F_c=\bigcup_iF^i_c$. A detailed proof can be found in \cite[Proof of Theorem 1]{CR10}.
We define
$$\mu_n=\inf\{c\in(0,n): \text{ the set $F_c$ is connex} \},\qquad u_n=v_n-\mu_n.$$
By definition, $u_n$ is a minimal solution on $\Omega$ which take the values $n-\mu_n$ on $A_i$ and $-\mu_n$ on $B_i$.
\begin{kd}
There exist two piecewise minimal solutions $u^+,u^-$ on $\Omega$ such that $u^-\le u_n\le u^+$ for every $n$.
\end{kd}
\begin{proof}
There exist, by the case \ref{cas2.2}, the minimal solutions $u_i^\pm$ on $\Omega$ such that$$u^+_i=\begin{cases}
\infty&\text{ on } \bigcup_{i'\ne i}A_{i'},\\
0&\text{ on } A_i\cup\bigcup_jB_j,
\end{cases}\qquad
u^-_i=\begin{cases}
-\infty&\text{ on } \bigcup_{i'\ne i}B_{i'},\\
0&\text{ on } B_i\cup\bigcup_jA_j.
\end{cases}
$$
Define $$u^+=\max_i u^+_i,\qquad u^-=\min_i u^-_i.$$
Observe that, by definition of $\mu_n$, both $E_{\mu_n}$ and $F_{\mu_n}$ are disconnected. In particular, for
every $i_1$, there exists an $i_2$ such that $E^{i_1}_{\mu_n}\cap E^{i_2}_{\mu_n}=\emptyset$ and we obtain, applying the maximum
principle,
$$0\le u_n|_{E^{i_1}_{\mu_n}}
\le u^+_{i_2}|_{E^{i_1}_{\mu_n}}.$$
Similarly, for every
$j_1$, there exists an $j_2$ such that $F^{j_1}_{\mu_n}\cap F^{j_2}_{\mu_n}=\emptyset$ and we obtain, applying the maximum
principle,
$$u^-_{j_2}|_{F^{j_1}_{\mu_n}}\le u_n|_{F^{j_1}_{\mu_n}}
\le0.$$
It follows that $u^-\le u_n\le u^+$ for every $n$.
\renewcommand{\qedsymbol}{$\Diamond$}
\end{proof}
By the previous assertion and the compactness theorem, Theorem \ref{TC}, there exists a subsequence $\{u_{\sigma(n)}\}$ of $\{u_n\}$ that converges on compact sets of $\Omega$ to a minimal solution $u$.
\begin{kd}
$$\lim_{n\to\infty}\mu_{\sigma(n)}=\infty,\qquad \lim_{n\to\infty} (n-\mu_{\sigma(n)})=\infty.$$
\end{kd}
\begin{proof}
Assume the contrary, that there exists a subsequence $\{\mu_{\sigma'(n)}\}$ of $\{\mu_{\sigma(n)}\}$ that converges to some $\mu_\infty$.
Then, by definition of $u$, that $u$ takes the values $\infty$ on $A_i$ and $-\mu_\infty$ on $B_i$. So, by the proof of necessity, $2a_\euc(\Omega)<\ell_\euc(\Omega)$, which contradicts with hypothesic \ref{cond1}. Then $\lim_{n\to\infty}\mu_{\sigma(n)}=\infty$. In the same way, we can show that $\lim_{n\to\infty} (n-\mu_{\sigma(n)})=\infty.$
\renewcommand{\qedsymbol}{$\Diamond$}
\end{proof}
So, by the previous assertion, we conclude $u$ takes $+\infty$ on $A_i$ and $-\infty$ on $B_i$.
\renewcommand{\qedsymbol}{$\heartsuit$}
\end{proof}
This completes the proof of the existence part of the theorem.
\end{proof}

The remainder of this section will be devoted to the proof of the uniqueness of Theorem \ref{J-S type1}.

\begin{dl}\label{PMNDM}{\rm (Maximum principle for unbounded domains with possible infinite boundary data)} Let $\Omega\subset\hh^2$ be a Scherk domain. Let $u_1,u_2$ be two solutions of type Jenkins-Serrin on $\Omega$. If the family $\{C_i\}$ is non-empty, assume that $\limsup (u_1-u_2)\le 0$ when ones approache to $\bigcup_iC_i$. If $\{C_i\}$ is empty, suppose that $u_1\le u_2$ at some point $p\in\Omega$. Then in either case $u_1\le u_2$ on $\Omega$.
\end{dl}

\begin{proof}
Assume the contrary, that the set $\{u_1>u_2\}$ is not empty.

Let $N$, $\varepsilon$ be two positive constants with $N$ large, $\varepsilon$ small. Define a function
$$\varphi=\left[u_1-u_2-\varepsilon\right]_{0}^{N-\varepsilon}=\begin{cases}
N-\varepsilon&u_1-u_2\ge N\\
u_1-u_2-\varepsilon& \varepsilon<u_1-u_2<N\\
0&u_1-u_2\le \varepsilon.
\end{cases}$$
Then $\varphi$ is Lipschitz and vanishes in a neighborhood of any point of $C_i$, $0\le \varphi<N$ and $\nabla\varphi=\nabla u_1-\nabla u_2$ on the set $\{\varepsilon<u_1-u_2<N \}$. Moreover, $\nabla\varphi=0$ almost everywhere in the complement of this set.
For each ideal vertex $p$ of $\Omega$, we take a net of geodesics $H_{p,n}$ that converges to $p$. Denote by $\Hcal_{p,n}$ the domain of $\hh^2$ delimited by $H_{p,n}$ such that the Euclidean mean convex vector of $H_{p,n}$ points interior.
Define
\begin{gather*}
\Omega_{n,\delta}=\Omega\setminus\left(\overline{\Dbb}_{\delta}(\dhr\Omega)\cup \bigcup_{p\in E_1}\overline \Dbb_{\frac{1}{n}}(p)\cup\bigcup_{p\in E_2}\overline{\Hcal}_{p,n}\right),\\ X_i'=\dhr\Omega_{n,\delta}\cap\overline{\Dbb}_\delta(X_i), \qquad (X\in\{A,B,C\})
\end{gather*}
and
$$\Gamma=\dhr\Omega_{n,\delta}\setminus\left(\bigcup_iA'_i\cup\bigcup_iB'_i\cup\bigcup_iC'_i\right),$$
where $E_1$ (resp. $E_2$) is the set of vertices (resp. ideal vertices) of $\Omega$ and $0<\delta=\delta(n)\ll\frac{1}{n}.$ (See Figure \ref{F13}).


\begin{figure}[!h]
\centering
\definecolor{ffqqqq}{rgb}{1,0,0}
\definecolor{qqqqff}{rgb}{0,0,1}
\begin{tikzpicture}[line cap=round,line join=round,>=triangle 45,x=1.0cm,y=1.0cm]
\clip(-2.02,-1.78) rectangle (8.12,5.66);
\draw (2.24,-1.02)-- (0,1);
\draw [shift={(1.8,1.2)}] plot[domain=1.46:3.25,variable=\t]({1*1.81*cos(\t r)+0*1.81*sin(\t r)},{0*1.81*cos(\t r)+1*1.81*sin(\t r)});
\draw [shift={(4.4,1.82)}] plot[domain=-1.17:1.82,variable=\t]({1*3.05*cos(\t r)+0*3.05*sin(\t r)},{0*3.05*cos(\t r)+1*3.05*sin(\t r)});
\draw [shift={(1.78,0.98)}] plot[domain=4.94:6.06,variable=\t]({1*2.05*cos(\t r)+0*2.05*sin(\t r)},{0*2.05*cos(\t r)+1*2.05*sin(\t r)});
\draw [shift={(0.91,2.19)}] plot[domain=-0.4:0.64,variable=\t]({1*1.35*cos(\t r)+0*1.35*sin(\t r)},{0*1.35*cos(\t r)+1*1.35*sin(\t r)});
\draw [shift={(3.43,4.87)}] plot[domain=4.06:4.76,variable=\t]({1*2.35*cos(\t r)+0*2.35*sin(\t r)},{0*2.35*cos(\t r)+1*2.35*sin(\t r)});
\draw (2.16,1.66)-- (3.54,2.52);
\draw [shift={(6.3,0.12)}] plot[domain=2.02:3.08,variable=\t]({1*1*cos(\t r)+0*1*sin(\t r)},{0*1*cos(\t r)+1*1*sin(\t r)});
\draw (2.24,-1.02)-- (2.82,0.48);
\draw [shift={(-1.54,-1)}] plot[domain=0:0.51,variable=\t]({1*3.78*cos(\t r)+0*3.78*sin(\t r)},{0*3.78*cos(\t r)+1*3.78*sin(\t r)});
\draw [shift={(2.73,1.96)}] plot[domain=4:4.77,variable=\t]({1*1.49*cos(\t r)+0*1.49*sin(\t r)},{0*1.49*cos(\t r)+1*1.49*sin(\t r)});
\draw [->,dash pattern=on 4pt off 4pt] (-2,-1) -- (8,-1);
\draw (7.18,5.24) node[anchor=north west] {$\mathbb{H}^2$};
\draw (3.48,1.98) node[anchor=north west] {$\Omega$};
\draw (0.4,0.4) node[anchor=north west] {$A_1$};
\draw (0.08,3.16) node[anchor=north west] {$C_1$};
\draw (7.24,1.04) node[anchor=north west] {$C_4$};
\draw (2,3)-- (3.66,4.78);
\draw (2.28,4.52) node[anchor=north west] {$B_1$};
\draw (3.78,0.52)-- (5.6,-0.98);
\draw (4.72,1.18)-- (5.3,0.18);
\draw (4.72,1.18)-- (5.86,1.02);
\draw [dash pattern=on 2pt off 2pt] (2,3) circle (0.32cm);
\draw [dash pattern=on 2pt off 2pt] (3.66,4.78) circle (0.32cm);
\draw [dash pattern=on 2pt off 2pt] (5.86,1.02) circle (0.32cm);
\draw [dash pattern=on 2pt off 2pt] (4.72,1.18) circle (0.32cm);
\draw [dash pattern=on 2pt off 2pt] (5.3,0.18) circle (0.32cm);
\draw [dash pattern=on 2pt off 2pt] (3.78,0.52) circle (0.32cm);
\draw [dash pattern=on 2pt off 2pt] (1.76,0.84) circle (0.32cm);
\draw [dash pattern=on 2pt off 2pt] (0,1) circle (0.32cm);
\draw [dash pattern=on 2pt off 2pt] (2.24,-1.02) circle (0.42cm);
\draw [dash pattern=on 2pt off 2pt] (5.6,-0.98) circle (0.42cm);
\draw [dash pattern=on 2pt off 2pt] (3.54,2.52) circle (0.32cm);
\draw [dash pattern=on 2pt off 2pt] (2.82,0.48) circle (0.32cm);
\draw [dash pattern=on 2pt off 2pt] (2.16,1.66) circle (0.32cm);
\draw [shift={(1.8,1.2)},dash pattern=on 1pt off 2pt on 4pt off 4pt,color=ffqqqq]  plot[domain=1.61:3.09,variable=\t]({1*1.63*cos(\t r)+0*1.63*sin(\t r)},{0*1.63*cos(\t r)+1*1.63*sin(\t r)});
\draw [shift={(-1.54,-1)},dash pattern=on 1pt off 2pt on 4pt off 4pt,color=ffqqqq]  plot[domain=0.1:0.44,variable=\t]({1*3.63*cos(\t r)+0*3.63*sin(\t r)},{0*3.63*cos(\t r)+1*3.63*sin(\t r)});
\draw [shift={(1.78,0.98)},dash pattern=on 1pt off 2pt on 4pt off 4pt,color=ffqqqq]  plot[domain=5.14:5.92,variable=\t]({1*1.89*cos(\t r)+0*1.89*sin(\t r)},{0*1.89*cos(\t r)+1*1.89*sin(\t r)});
\draw [shift={(0.91,2.19)},dash pattern=on 1pt off 2pt on 4pt off 4pt,color=ffqqqq]  plot[domain=-0.19:0.43,variable=\t]({1*1.17*cos(\t r)+0*1.17*sin(\t r)},{0*1.17*cos(\t r)+1*1.17*sin(\t r)});
\draw [shift={(3.43,4.87)},dash pattern=on 1pt off 2pt on 4pt off 4pt,color=ffqqqq]  plot[domain=4.19:4.64,variable=\t]({1*2.21*cos(\t r)+0*2.21*sin(\t r)},{0*2.21*cos(\t r)+1*2.21*sin(\t r)});
\draw [shift={(2.73,1.96)},dash pattern=on 1pt off 2pt on 4pt off 4pt,color=ffqqqq]  plot[domain=4.19:4.57,variable=\t]({1*1.3*cos(\t r)+0*1.3*sin(\t r)},{0*1.3*cos(\t r)+1*1.3*sin(\t r)});
\draw [shift={(6.3,0.12)},dash pattern=on 1pt off 2pt on 4pt off 4pt,color=ffqqqq]  plot[domain=2.34:2.76,variable=\t]({1*0.86*cos(\t r)+0*0.86*sin(\t r)},{0*0.86*cos(\t r)+1*0.86*sin(\t r)});
\draw [shift={(4.4,1.82)},dash pattern=on 1pt off 2pt on 4pt off 4pt,color=ffqqqq]  plot[domain=-1.03:1.72,variable=\t]({1*2.91*cos(\t r)+0*2.91*sin(\t r)},{0*2.91*cos(\t r)+1*2.91*sin(\t r)});
\draw [dash pattern=on 1pt off 2pt on 4pt off 4pt,color=ffqqqq] (2.3,3.12)-- (3.56,4.48);
\draw [dash pattern=on 1pt off 2pt on 4pt off 4pt,color=ffqqqq] (5.03,1.28)-- (5.59,1.18);
\draw [dash pattern=on 1pt off 2pt on 4pt off 4pt,color=ffqqqq] (4.76,0.86)-- (5.03,0.36);
\draw [dash pattern=on 1pt off 2pt on 4pt off 4pt,color=ffqqqq] (4.1,0.48)-- (5.41,-0.6);
\draw [dash pattern=on 1pt off 2pt on 4pt off 4pt,color=ffqqqq] (2.86,0.16)-- (2.53,-0.71);
\draw [dash pattern=on 1pt off 2pt on 4pt off 4pt,color=ffqqqq] (2.48,1.68)-- (3.39,2.24);
\draw [dash pattern=on 1pt off 2pt on 4pt off 4pt,color=ffqqqq] (0.32,0.94)-- (2.07,-0.64);
\draw [->,dash pattern=on 4pt off 4pt] (-1,-1.72) -- (-0.96,5.48);
\draw (7.7,-0.4) node[anchor=north west] {$x$};
\draw (-0.9,5.68) node[anchor=north west] {$y$};
\begin{scriptsize}
\draw [fill=qqqqff] (2.24,-1.02) circle (1.5pt);
\draw [fill=qqqqff] (0,1) circle (1.5pt);
\draw [fill=qqqqff] (2,3) circle (1.5pt);
\draw [fill=qqqqff] (5.6,-0.98) circle (1.5pt);
\draw [fill=qqqqff] (3.66,4.78) circle (1.5pt);
\draw [fill=qqqqff] (3.78,0.52) circle (1.5pt);
\draw [fill=qqqqff] (2.16,1.66) circle (1.5pt);
\draw [fill=qqqqff] (3.54,2.52) circle (1.5pt);
\draw [fill=qqqqff] (4.72,1.18) circle (1.5pt);
\draw [fill=qqqqff] (5.86,1.02) circle (1.5pt);
\draw [fill=qqqqff] (5.3,0.18) circle (1.5pt);
\draw [fill=qqqqff] (2.82,0.48) circle (1.5pt);
\draw [fill=black] (-1,-1) circle (1.5pt);
\draw[color=black] (-0.86,-0.72) node {$O$};
\draw [fill=qqqqff] (1.76,0.84) circle (1.5pt);
\end{scriptsize}
\end{tikzpicture}
\caption{The domain $\Omega_{n,\delta}$}\label{F13}
\end{figure}
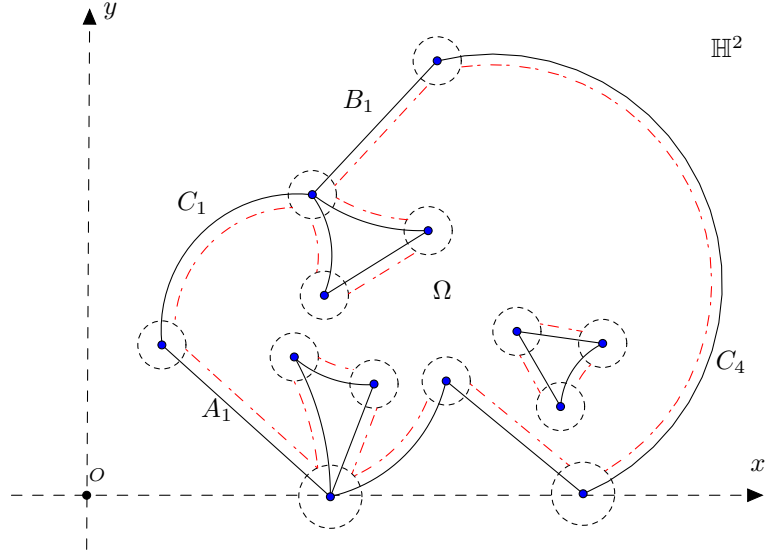

Define
$$J_n=\int_{\dhr\Omega_{n,\delta}}\varphi y\vh{X_{u_1}-X_{u_2},\nu}\,\dsf s,$$
where $\nu$ is the exterior normal to $\dhr\Omega_{n,\delta}$.
\begin{kd}
\begin{itemize}
\item[\rm (i)] $J_n\ge 0$, equality if and only if $\nabla u_1=\nabla u_2$ on the set $\{x\in\Omega_{n,\delta}:\varepsilon<u_1-u_2<N\}$.
\item[\rm (ii)] $J_n$ is increasing as $n\to \infty$.
\end{itemize}
\end{kd}
\begin{proof}
By Divergence theorem, we have
\begin{align*}
J_n&=\int_{\Omega_{n,\delta}}\Div\left( \varphi y(X_{u_1}-X_{u_2})\right)\,\dsf \Acal\\
&=\int_{\Omega_{n,\delta}}\langle y\nabla\varphi,X_{u_1}-X_{u_2}\rangle\,\dsf \Acal+\int_{\Omega_{n,\delta}}\varphi\Div(yX_{u_1}-yX_{u_2})\,\dsf \Acal.
\end{align*}
By the hypotheses, we obtain
$$\varphi\Div(yX_{u_1}-yX_{u_2})=\varphi(\Mcal u_1-\Mcal u_2)= 0.$$
Moreover, by Lemma \ref{lem cle}
$$\langle y\nabla\varphi,X_{u_1}-X_{u_2}\rangle=\vh{y\nabla u_1-y\nabla u_2,\frac{y\nabla u_1}{W_{u_1}}-\frac{y\nabla u_2}{W_{u_2}}}\ge 0.$$

\renewcommand{\qedsymbol}{$\Diamond$}
\end{proof}

\begin{kd}
$J_n=o(1)$ as $n\to\infty$.
\end{kd}
\begin{proof}
We have
\begin{align*}
J_n=&\sum_i\int_{A'_i}\varphi y\vh{X_{u_1}-X_{u_2},\nu}\,\dsf s+
\sum_i\int_{B'_i}\varphi y\vh{X_{u_1}-X_{u_2},\nu}\,\dsf s\\
&+\sum_i\int_{C'_i}\varphi y\vh{X_{u_1}-X_{u_2},\nu}\,\dsf s+\int_{\Gamma}\varphi y\vh{X_{u_1}-X_{u_2},\nu}\,\dsf s.
\end{align*}
Since $\varphi=0$ on a neighborhood of $\bigcup_iC_i$, we have
$$\sum_i\int_{C_i}\varphi y\vh{X_{u_1}-X_{u_2},\nu}\,\dsf s=0.$$
Moreover, since $\td{X_{u_i}}\le 1,i=1,2;$ $\varphi\le N$
$$\left|\int_{\Gamma}\varphi y\vh{X_{u_1}-X_{u_2},\nu}\,\dsf s\right|\le 2N\ell_\euc(\Gamma)=o(1).$$
By Lemma \ref{lem2.5}, then
$$\int_{A'_i}\varphi y\vh{X_{u_1}-X_{u_2},\nu}\,\dsf s\le N(\ell_\euc(A'_i)-F_{u_2}(A'_i)),$$
$$\int_{B'_i}\varphi y\vh{X_{u_1}-X_{u_2},\nu}\,\dsf s\le N(\ell_\euc(B'_i)+F_{u_1}(B'_i)).$$
For every $i$ and $X\in\{A,B\}$, denote by $\Omega^X_i$ a component of $\Omega\setminus\Omega_{n,\delta}$ such that $X'_i\subset\dhr\Omega^X_i$ and define $X''_i=\dhr\Omega^X_i\cap X_i$.
By Lemma \ref{lem2.5}, we have
$$0=F_u\Ngoac{\dhr\Omega^X_i}=F_u(X''_i)-F_u(X'_i)+F_u\Ngoac{\dhr\Omega^X_i\setminus(X'_i\cup X''_i)},$$
$$F_u(X''_i)=\begin{cases}
\ell_\euc(A''_i)=\ell_\euc(A'_i)+o\left(1\right)& \text{ si } X_i=A_i,\\
\ell_\euc(B''_i)=-\ell_\euc(B'_i)+o\left(1\right)& \text{ si } X_i=B_i
\end{cases}, $$
$$\left|F_u(\dhr\Omega^X_i\setminus(X'_i\cup X''_i))\right|\le \ell_{\euc}\left(\dhr\Omega^X_i\setminus(X'_i\cup X''_i)\right)=o(1),$$
where $u\in\{u_1,u_2\}$. So
$$\ell_\euc(A'_i)-F_{u_2}(A'_i)=o(1),\qquad \ell_\euc(B'_i)+F_{u_1}(B'_i)=o(1).$$
This proves the assertion.
\renewcommand{\qedsymbol}{$\Diamond$}
\end{proof}

It follows from the previous assertions that $\nabla u_1=\nabla u_2$ on the set $\{\varepsilon<u_1-u_2<N \}$. Since $\varepsilon>0$ and $N$ are arbitrary, $\nabla u_1=\nabla u_2$ whenever $u_1> u_2$. It follows that $u_1=u_2+c, (c>0)$ in any nontrivial component of the set $\{u_1>u_2\}$. Then the maximum principle, Theorem \ref{PM}, ensures  $u_1=u_2+c$ in $\Omega$ and by the assumptions of the theorem, the constant must nonpositive, a contradiction.
\end{proof}

\subsection*{Acknowledgements}

This is part of my Ph.D. thesis, written  at Universit\'e Paul Sabatier, Toulouse. 
I wish to thank Laurent Hauswirth for drawing my attention to the problem of type Jenkins-Serrin in $\sol$. 
I would like to express my deep gratitude to Pascal Collin and Laurent Hauswirth for many helpful suggestions during the preparation of the paper.

\end{document}